\documentclass[11pt, reqno]{amsart}

\usepackage{amsfonts}
\usepackage{amssymb}
\usepackage{amsthm}
\usepackage{amsmath}
\usepackage{a4wide}
\usepackage{mathrsfs}
\usepackage{epsfig}
\usepackage{palatino}
\usepackage{tikz,fp,ifthen}
\usepackage{float}
\usepackage{graphicx,cite}
\usepackage{esint}
\usepackage{footnote}
\usepackage{verbatim}
\usepackage{todonotes}
\usepackage{xstring}

\usepackage{graphicx}
\usepackage{color}
\definecolor{citegreen}{rgb}{0,0.6,0}
\definecolor{refred}{rgb}{0.8,0,0}
\usepackage{hyperref}

\usepackage{mathtools}
\mathtoolsset{showonlyrefs}

\usepackage{enumitem}

\usepackage[]{draftwatermark}
\SetWatermarkText{${\mathcal{DRAFT}}$}
\SetWatermarkScale{6.2}

\theoremstyle{plain}
\newtheorem{thm}{Theorem}[section]

\theoremstyle{definition}

\theoremstyle{remark}
\newtheorem{remark}[thm]{Remark}

\numberwithin{equation}{section}

\def\comp{\circ}

\newcommand{\Lebes}{\mathscr{L}}

\def\eps{\varepsilon}

\def\Riem{{\mathrm {Riem}}}

\def\RRR{{\mathrm R}}

\def\BBB{{\mathrm B}}
\def\HHH{{\mathrm H}}

\def\R{\mathbb R}

\def\R{{{\mathbb R}}}

\def\N{\mathbb N}

\def\T{\mathbb T}

\def\Div{\operatorname*{div}\nolimits}

\newcommand{\pa}{\partial}
\newcommand{\intbar}{\etaathop{\int\etaakebox(-13.5,0){\rule[4pt]{.7em}{0.3pt}}
\kern-6pt}\nolimits}

\newcommand{\p}{\varphi}
\newcommand{\grad}{\nabla}

\newcommand{\Id}{\mathrm{Id}}

\DeclarePairedDelimiter{\norma}{\lVert}{\rVert}

\usetikzlibrary{backgrounds}
\usetikzlibrary{decorations.pathmorphing,backgrounds,fit,calc,through}
\usetikzlibrary{arrows}
\usetikzlibrary{shapes,decorations,shadows}
\usetikzlibrary{fadings}
\usetikzlibrary{patterns}
\usetikzlibrary{mindmap}
\usetikzlibrary{decorations.text}
\usetikzlibrary{decorations.shapes}

\begin{document}

\title[Uniform {S}obolev, interpolation and geometric {Calder\'on--Zygmund} inequalities]{Uniform {S}obolev, interpolation and geometric {Calder\'on--Zygmund} inequalities for graph hypersurfaces}
\author{Serena Della Corte}
\address{Serena Della Corte\\
Delft Institute of Applied Mathematics (DIAM), Delft University of Technology, The Netherlands}
\email{s.dellacorte@tudelft.nl}

\author{Antonia Diana}
\address{Antonia Diana\\
Scuola Superiore Meridionale, Napoli, Italy}
\email{antonia.diana@unina.it}

\author{Carlo Mantegazza}
\address{Carlo Mantegazza\\
Dipartimento di Matematica e Applicazioni ``Renato Caccioppoli'', Universit\`a di Napoli Federico II \& Scuola Superiore Meridionale, Napoli, Italy}
\email{carlo.mantegazza@unina.it}

\date{\today}

\begin{abstract} In this note, our aim is to show that families of smooth hypersurfaces of $\R^{n+1}$ which are all ``$C^1$--close'' enough to a fixed compact, embedded one, have uniformly bounded constants in some relevant inequalities for mathematical analysis, like Sobolev, Gagliardo--Nirenberg and ``geometric'' Calder\'on--Zygmund inequalities. 
\end{abstract}

\maketitle
\setcounter{tocdepth}{1}
\tableofcontents

\section{Introduction and preliminaries}\label{intro}

In this note, our aim is to show that families of smooth hypersurfaces of $\R^{n+1}$ which are all ``$C^1$--close'' enough to a fixed compact, embedded one, have uniformly bounded constants in some relevant inequalities for mathematical analysis, like Sobolev, Gagliardo--Nirenberg, ``geometric'' Calder\'on--Zygmund, trace and extension inequalities.  These technical results are quite useful, in particular, in the study of the geometric flows of hypersurfaces, when one studies the behavior of the hypersurfaces ``close'' (in some norm, for instance in $C^1$--norm) to critical ones (possibly ``stable'') or the asymptotic limits of flows existing for all times (see for instance~\cite{AcFuMoJu,AcFuMo,DDMsurvey,DFM}, where such controls on the constants are necessary).

\smallskip

We start by setting up some notation and recall some basic facts about hypersurfaces in Euclidean spaces that we need in the sequel, possible 
references are~\cite{benedetti1,abate,gahula}.

We will consider smooth, compact hypersurfaces $M$, embedded in $\R^{n+1}$, getting a Riemannian metric $g$ by pull--back of the standard scalar product $\langle\cdot\,\vert\,\cdot\rangle$ of $\R^{n+1}$ via the embedding map $\varphi:M\to\R^{n+1}$, hence, turning it into a Riemannian manifold $(M,g)$. Then, we use $\nabla$ for the associated Levi--Civita covariant derivative and $\mu$ for the canonical measure induced by the metric $g$, which actually coincides with the $n$--dimensional Hausdorff measure ${\mathcal{H}}^n$ of $\R^{n+1}$ restricted to $M$.
Then, the components of $g$ in a local chart are
$$
g_{ij}=\left \langle\frac{\pa \varphi}{\pa x_i}\,\Big \vert \frac{\pa \varphi}{\pa x_j}\right \rangle
$$
and the ``canonical'' measure $\mu$, induced on $M$ by the metric $g$ is then locally described by $\mu=\sqrt{\det g_{ij}}\,{\Lebes}^{n}$, where ${\Lebes}^{n}$ is the standard Lebesgue measure on $\R^{n}$.\\
The inner product on $M$, extended to tensors, is given by
$$
g(T , S)=g_{i_1s_1}\dots g_{i_k s_k}g^{j_1
z_1}\dots g^{j_l z_l} T^{i_1\dots
 i_k}_{j_1\dots j_l}S^{s_1\dots
 s_k}_{z_1\dots z_l}
$$
where $g_{ij}$ is the matrix of the coefficients of the metric
tensor in the local coordinates and $g^{ij}$ is its 
inverse. Clearly, the norm of a tensor is then
$$
\vert T\vert=\sqrt{g(T , T)}\,.
$$

The induced Levi--Civita covariant derivative on $(M,g)$ of a vector field $X$ and of a 1--form $\omega$ 
are respectively given by
\begin{equation}\label{Levi}
\nabla _jX^i=\frac{\partial X^i}{\partial
 x_j}+\Gamma^{i}_{jk}X^k\,, \qquad \nabla _j\omega_i=\frac{\partial \omega_i}{\partial
 x_j}-\Gamma^k_{ji}\omega_k\,, 
\end{equation}
where $\Gamma^{i}_{jk}$ are the Christoffel symbols of the connection $\nabla$, expressed
by the formula
\begin{equation}\label{Chr}
\Gamma^{i}_{jk}=\frac{1}{2} g^{il}\Bigl(\frac{\partial\,}{\partial
 x_j}g_{kl}+\frac{\partial\,}{\partial
 x_k}g_{jl}-\frac{\partial\,}{\partial
 x_l}g_{jk}\Bigr)\,.
\end{equation}
With $\nabla^m T$ we will mean the $m$--th iterated covariant
derivative of a tensor $T$.

\smallskip

Being $M$ embedded, we can assume it is a subset of $\R^{n+1}$ (hence the embedding map is the identity) and we denote with $\nu:M\to\R^{n+1}$ its global unit normal vector field, {\em pointing outward}. It is indeed well known (theorem of Jordan--Brouwer, see~\cite[Proposition~12.2]{benedetti1}, for instance) that any compact, embedded $M$ ``divides'' $\R^{n+1}$ in two connected components, one of them bounded (called ``the interior''), both having $M$ as its smooth boundary, hence the hypersurface is orientable and such field $\nu$ exists.

Then, we define the {\em second fundamental form} $\BBB$ which is a symmetric $2$--form given, in a local chart, by its components
\begin{equation}\label{secform}
\BBB_{ij} = - \,\biggl \langle \frac{\pa^2 \varphi}{\pa x_i \pa x_j}\,\biggr\vert \,\nu \biggr \rangle
\end{equation}
and whose trace is the {\em mean curvature} $\HHH= g^{ij} \BBB_{ij}$ of the hypersurface (with these choices, the standard sphere of $\R^{n+1}$ has positive mean curvature).

\begin{remark}\label{graph}
If the hypersurface $M$ is locally the graph of a function $f:U\to\R$ with $U$ an open subset of $\R^{n}$, that is, $M=\{(x,f(x))\ :\ x\in U\}$, then we have 
\begin{equation}\label{metric}
g_{ij}=\delta_{ij}+\frac{\partial f}{\partial x_i}\frac{\partial f}{\partial x_j}\,,\qquad\qquad \nu=-\frac{(\nabla^{\R^n}\! f,-1)}{\sqrt{1+\vert\nabla^{\R^n}\! f\vert^2}}\, ,
\end{equation}
\begin{equation}\label{secondform}
\BBB_{ij}=-\frac{{\mathrm {Hess}}_{ij}^{\R^n}\!f}{\sqrt{1+\vert\nabla^{\R^n}\! f\vert^2}}\, ,
\end{equation}
\begin{equation}\label{meancurvature}
\HHH=-\frac{\Delta^{\R^n}\!f}{\sqrt{1+\vert\nabla^{\R^n}\! f\vert^2}}
+\frac{{\mathrm {Hess}}^{\R^n}\!f(\nabla^{\R^n}\! f,\nabla^{\R^n}\! f)}{\big(\sqrt{1+\vert\nabla^{\R^n}\! f\vert^2}\big)^3}=-
\Div^{\R^n}\!\biggl(\frac{\nabla^{\R^n}\! f}{\sqrt{1+\vert\nabla^{\R^n}\! f\vert^2}}\biggr)
\end{equation}
where ${\mathrm {Hess}}^{\R^n}\!f$ is the (standard) Hessian of the function $f$.
\end{remark}

Then, the following {\em Gauss--Weingarten relations} hold,
\begin{equation}\label{GW}
\frac{\pa^2\varphi}{\pa x_i\pa
 x_j}=\Gamma_{ij}^k\frac{\pa\varphi}{\pa
 x_k}- \BBB_{ij}\nu\qquad\qquad\frac{\pa\nu}{\pa x_j}=
\BBB_{jl}g^{ls}\frac{\pa\varphi}{\pa x_s}\,,
\end{equation}
which easily imply 
\begin{equation}
\nabla^2\varphi=-\BBB\nu\qquad\text{ and }\qquad\Delta\varphi=-\HHH\nu\,.\label{lap}
\end{equation}

The symmetry properties of the covariant derivative of $\BBB$ are
given by the following Codazzi equations,
\begin{equation}\label{codaz0}
\nabla_i\BBB_{jk}=\nabla_j\BBB_{ik}=\nabla_k\BBB_{ij}
\end{equation}
which imply the following {\em Simons' identity} (see~\cite{simons}),
\begin{equation}\label{codaz}
\Delta \BBB_{ij}=\nabla_i\nabla_j\HHH + \HHH \, 
\BBB_{il}g^{ls}\BBB_{sj}-\vert\BBB\vert^2\BBB_{ij}\,.
\end{equation}

Finally, the Riemann tensor can be expressed as ({\em Gauss equations}),
\begin{equation}\label{Gauss-eq}
\RRR_{ijkl}\,=\,\BBB_{ik}\BBB_{jl}-\BBB_{il}\BBB_{jk}\,.
\end{equation}

\smallskip

If now we choose a fixed smooth, compact, embedded hypersurface $M_0$ of $\R^{n+1}$, it is well known (by its compactness and smoothness) that, for $\eps>0$ small enough, $M_0$ has a {\em tubular neighborhood} 
\begin{equation}\label{tubdef}
N_\eps=\bigl\{x \in \R^{n+1} \, : \, d(x,M_0)<\eps\bigr\}
\end{equation}
(where $d$ is the Euclidean distance on $\R^{n+1}$) such that the {\em orthogonal projection map} $\pi:N_\eps\to M_0$ giving the (unique) closest point on $M_0$, is well defined and smooth. Then, if $E$ is ``the interior'' of $M_0$, the {\em signed distance function} $d_E:N_\eps\to\R$ from $M_0$
\begin{equation}\label{sign dist}
d_E(x)=
\begin{cases}
d(x, M_0) &\text{if $x \notin E$}\\
-d(x, M_0) &\text{if $x \in E$}
\end{cases} 
\end{equation}
is smooth in $N_\eps$ and $\nu(x)=\nabla^{\R^{n+1}}\! d_E(x)$, for every $x\in M_0$. Moreover, for every $x\in N_\eps$, the projection map $\pi$ is given explicitly by 
\begin{equation}\label{eqcar2050}
\pi_E(x)=x-\nabla^{\R^{n+1}}\! d^2_E(x)/2=x-d_E(x)\nabla^{\R^{n+1}}\! d_E(x)
\end{equation}
(indeed, actually $\nabla^{\R^{n+1}}\! d_E(x)=\nabla^{\R^{n+1}}\! d_E(\pi_E(x))=\nu(\pi_E(x))$ for every $x\in N_\eps$).

From now on, we will consider smooth hypersurfaces contained in $N_\eps$ that can be written (possibly after reparametrization) as graph over $M_0$, that is,
\begin{equation}\label{psidescr}
M=\bigl\{x+\psi(x)\nu(x) \, : \, x\in M_0\bigr\},
\end{equation}
for a smooth ``height function'' $\psi:M_0\to\R$ with $|\psi(x)|<\eps$, for every $x\in M_0$.\\
We define the following families (clearly all containing $M_0$),
\begin{equation}\label{psidescr2}
\mathfrak{C}^1_\delta(M_0)=\Bigl\{M=\bigl\{x+\psi(x)\nu(x) \, : \, x\in M_0\bigr\}\,\, \text{for a smooth $\psi:M_0\to\R$ with $\norma{\psi}_{C^1(M_0)}<\delta$}\,\Bigr\}
\end{equation}
where $\delta\in(0,\eps)$ and we are considering on $M_0$ the induced metric from $\R^{n+1}$ (in order to define $|d\psi|$). Sometimes, we will use the expression ``$C^1$--close to $M_0$'', meaning that the above constant $\delta$ is small. Moreover, since we will use it, we also define the subfamily $\mathfrak{C}^{1,\alpha}_\delta(M_0)$ of the hypersurfaces $M\in\mathfrak{C}^1_\delta(M_0)$ such that the ``height function'' satisfies $\norma{\psi}_{C^{1,\alpha}(M_0)}<\delta$.

We are going to see that the constants in Sobolev, Gagliardo--Nirenberg, some geometric Calder\'on--Zygmund inequalities, trace and extension inequalities are uniformly bounded, depending only on $M_0$ and $\delta$.

Before starting discussing that, we introduce another technical construction. 
We notice that, possibly choosing a smaller $\eps>0$, the tubular neighborhood $N_\eps$ of $M_0$ defined above, can be covered by a finite number of open hypercubes $Q_1,\dots,Q_k\subseteq\R^{n+1}$ respectively centered at some points $p_1,\dots,p_k\in M_0$, such that, for every $i\in\{1,\dots,k\}$ and every $M\in\mathfrak{C}^1_\delta(M_0)$, with $\delta\in(0,\varepsilon)$, the ``pieces'' of hypersurfaces $M\cap Q_i$ can be written as {\em orthogonal} graphs on the affine hyperplanes $\Pi_{p_i}M_0=p_i+T_{p_i}M_0$, parallel to the tangent hyperplanes to $M_0$ at the points $p_i\in M_0$ and passing through them, as in the following figure.
\begin{figure}[H]
\centering
\includegraphics[scale=1.1, clip]{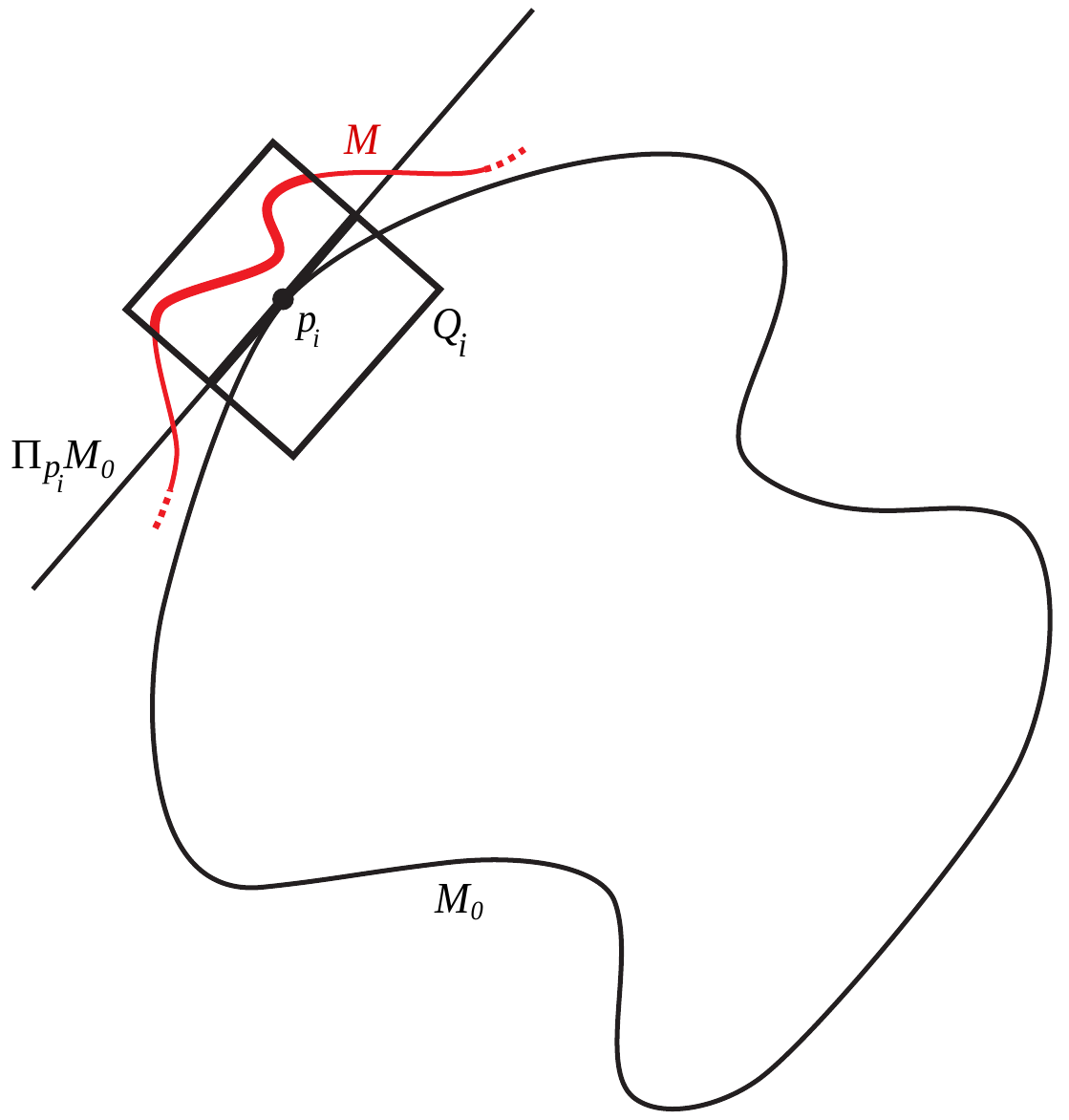}
\label{cut11fig}
\end{figure}
\noindent Then, we let $\rho_i:\R^{n+1}\to[0,1]$ a smooth partition of unity (with compact support) for $N_\eps$, associated to the open covering $Q_i$, hence, if $M\in\mathfrak{C}^1_\delta(M_0)$ and $u:M\to\R$, there holds
$$
u(y)=\sum_{i=1}^ku(y)\rho_i(y)
$$
with the compact support of $u\rho_i:M\to\R$ contained in the piece $M\cap Q_i$ of the hypersurface $M$, which is described as the graph of a smooth function $\theta_i:\Pi_{p_i}M_0\to\R$, that is, $M\cap Q_i$ is the image of the map $x\mapsto\Theta(x)=x+\theta_i(x)\nu(p_i)$ on $\Pi_{p_i}M_0\cap Q_i$. Moreover, it is easy to see that, possibly choosing an even smaller $\varepsilon>0$, we have $\Vert\theta_i\Vert_{C^1(\Pi_{p_i}M_0)}\leqslant2\delta$, for every $i\in\{1,\dots,k\}$, since also $M_0$ can be locally written as an orthogonal graph on $\Pi_{p_i}M_0$.

We notice and underline that the family (and the number) of the hypercubes $Q_i$, as well as the width $\eps>0$ of the tubular neighborhood $N_\eps$ that we considered for this construction, only depend on $M_0$, precisely on its local and global geometry (in particular, on its second fundamental form $\BBB_0$ -- see~\cite{breun3} for more details).

\smallskip

{\em We highlight to the reader that in the following, we will often denote with $C$ a constant which may vary from a line to another.}

\section{Sobolev, Poincar\'e and Gagliardo--Nirenberg interpolation inequalities}\label{SobSec}

We start discussing the Sobolev constants $C_S(M,p)$ of any compact $n$--dimensional hypersurface $M$, for every $p\in[1,n)$, entering in the following inequalities (which are known to hold, see~\cite[Chapter~2]{aubin0}, for instance),
\begin{equation}\label{sob1}
\Vert u\Vert_{L^{p^*}(M)}=
\Bigl(\int_M\vert u\vert^{p^*}\,d\mu\Bigr)^{1/p^*}\leqslant C_S(M,p)\Bigl(\int_M\vert \nabla u\vert^p+\vert u\vert^p\,d\mu\Bigl)^{1/p}=C_S(M,p)\Vert u\Vert_{W^{1,p}(M)}
\end{equation}
for every $C^1$--function $u:M\to\R$ (or $u\in W^{1,p}(M)$), where $p^*=\frac{np}{n-p}$ is the {\em Sobolev conjugate exponent} of $p$. It is well known that a bound on $C_S(M,1)$ implies a bound on $C_S(M,p)$, for every $p\in[1,n)$ (see~\cite[Chapter~2, Section~5]{aubin0}, for instance), hence we concentrate on the case $p=1$, where $1^*=\frac{n}{n-1}$.

We first want to argue localizing things by means of the construction of the previous section. We then have a finite family of hypercubes $Q_i$ centered at $p_i\in M_0$, the partition of unity $\rho_i$ and a parametrization $x\mapsto\Theta(x)=x+\theta_i(x)\nu_i$ on $\Pi_{p_i}M_0\cap Q_i$ of each piece $M\cap Q_i$ of any smooth hypersurface $M\in\mathfrak{C}^1_\delta(M_0)$, where $\nu_i=\nu(p_i)$ and the functions $\theta_i:\Pi_{p_i}M_0\to\R$ satisfy $\Vert\theta_i\Vert_{C^1(\Pi_{p_i}M_0)}\leqslant2\delta$, for every $i\in\{1,\dots,k\}$. Moreover, in dealing with any piece $M\cap Q_i$, we will assume (without clearly losing generality) that $\Pi_{p_i}M_0=\R^n\subseteq\R^{n+1}$ and we observe that in such parametrization, by formula~\eqref{metric}, the Riemannian measure $\mu$ associated to the (induced) metric $g$ on $M$ is given by $\mu=J\Theta\,\mathscr{L}^n$, with $\mathscr{L}^n$ the Lebesgue measure on $\Pi_{p_i}M_0=\R^n$ and $J\Theta=\sqrt{1+\vert\nabla^{\R^n}\!\theta_i\vert^2\,}$, which clearly satisfies $1\leqslant J\Theta\leqslant 1+2\delta$.

For every $C^1$--function $u:M\to\R$, we can write
$$
\Bigl(\int_M\vert u\vert^{\frac{n}{n-1}}\,d\mu\Bigr)^{\frac{n-1}{n}}=\Bigl(\int_{M}\Bigl\vert\sum_{i=1}^ku\rho_i\Bigr\vert^{\frac{n}{n-1}}\,d\mu\Bigr)^{\frac{n-1}{n}}\leqslant\sum_{i=1}^k\Bigl(\int_{M\cap Q_i}\vert u\rho_i\vert^{\frac{n}{n-1}}\,d\mu\Bigr)^{\frac{n-1}{n}}
$$
as the compact support of $u\rho_i$ is contained in $M\cap Q_i$.\\
Then, for every $C^1$ function $v:M\to\R$ with compact support in $M\cap Q_i$, there holds
\begin{align*}
\Bigl(\int_{M\cap Q_i}\vert v(y)\vert^{\frac{n}{n-1}}\,d\mu(y)\Bigr)^{\frac{n-1}{n}}
=&\,\Bigl(\int_{\R^n}\vert v(x+\theta_i(x)\nu_i)\vert^{\frac{n}{n-1}}J\Theta(x)\,dx\Bigr)^{\frac{n-1}{n}}\\
\leqslant&\,C(\delta)\Bigl(\int_{\R^n}\vert v(x+\theta_i(x)\nu_i)\vert^{\frac{n}{n-1}}\,dx\Bigr)^{\frac{n-1}{n}}\,,
\end{align*}
as $J\Theta\leqslant 1+2\delta$ and applying the Sobolev inequality for functions with compact support in $\R^n$, we have
\begin{align}
\Bigl(\int_{\R^n}\vert v(x+\theta_i(x)\nu_i)\vert^{\frac{n}{n-1}}\,dx\Bigr)^{\frac{n-1}{n}}
\leqslant&\,C\int_{\R^n}\vert \nabla^{\R^n}\! [v(x+\theta_i(x)\nu_i)]\vert\,dx\\
=&\,C\int_{\R^n}\bigl\vert \nabla v(x+\theta_i(x)\nu_i)\comp\bigl(\Id+\nabla^{\R^n}\!\theta_i(x)\otimes\nu_i\bigr)\bigr\vert\,dx\\
\leqslant&\,C\int_{\R^n}\vert \nabla v(x+\theta_i(x)\nu_i)\vert\,\bigl\vert\Id+\nabla^{\R^n}\!\theta_i(x)\otimes\nu_i\bigr\vert\,dx\\
=&\,C\int_{\R^n}\vert \nabla v(x+\theta_i(x)\nu_i)\vert\,\sqrt{1+\vert\nabla^{\R^n}\!\theta_i\vert^2\,}\,dx\\
=&\,C\int_M\vert \nabla v(y)\vert\,d\mu(y)\,,\label{tangent}
\end{align}
as $\sqrt{1+\vert\nabla^{\R^n}\!\theta_i\vert^2\,}=J\Theta$. Hence,
$$
\Bigl(\int_{M\cap Q_i}\vert v\vert^{\frac{n}{n-1}}\,d\mu\Bigr)^{\frac{n-1}{n}}\leqslant 
C(\delta)\int_M\vert \nabla v\vert\,d\mu
$$
and setting $v_i=u\rho_i$, after summing on $i\in\{1,\dots,k\}$, we conclude
\begin{align}
\Bigl(\int_M\vert u\vert^{\frac{n}{n-1}}\,d\mu\Bigr)^{\frac{n-1}{n}}
\leqslant&\,\sum_{i=1}^k\Bigl(\int_{M\cap Q_i}\vert v_i\vert^{\frac{n}{n-1}}\,d\mu\Bigr)^{\frac{n-1}{n}}\\
\leqslant&\,C(\delta)\sum_{i=1}^k\int_M\vert \nabla v_i\vert\,d\mu\\
=&\,C(\delta)\sum_{i=1}^k\int_M\vert \nabla u\vert\rho_i+\vert u\vert\,\vert \nabla\rho_i\vert\,d\mu\\
\leqslant&\,C(\delta)\int_M\vert \nabla u\vert\,d\mu+C(M_0,\delta)\int_M\vert u\vert\,d\mu\,,
\label{sobolev1}
\end{align}
as $\vert\nabla\rho_i\vert\leqslant C(M_0,\delta)$, for every $i\in\{1,\dots,k\}$. This clearly gives a uniform bound on $C_S(M,1)$ for all the hypersurfaces in $\mathfrak{C}^1_\delta(M_0)$, depending only on $M_0$ (in particular, on its second fundamental form $\BBB_0$, as we said in the previous section) and $\delta>0$.

\smallskip

Let now see an alternate line, based on the ``global'' graph representation of the hypersurfaces $M\in\mathfrak{C}^1_\delta(M_0)$ over $M_0$.\\
For every $C^1$ function $u:M\to\R$, we have
\begin{equation*}
\Bigl(\int_M\vert u(y)\vert^{\frac{n}{n-1}}\,d\mu(y)\Bigr)^{\frac{n-1}{n}}
=\Bigl(\int_{M_0}\vert u(x+\psi(x)\nu(x))\vert^{\frac{n}{n-1}}\,J\Psi(x)\,d\mu_0(x)\Bigr)^{\frac{n-1}{n}}
\end{equation*}
where $J\Psi$ is the Jacobian of the map $\Psi:M_0\to M$ and it is an easy check that, at every point $x\in M_0$, there holds
\begin{equation}\label{stimaJPsi}
\frac{1}{C(\BBB_0,\delta)}\leqslant J\Psi\leqslant C(\BBB_0,\delta)\,,
\end{equation}
for some constant $C(\BBB_0,\delta)>0$, where $\BBB_0$ is the second fundamental form of $M_0$. Moreover, $C(\BBB_0,\delta)$ goes to 1 as $\delta\to0$. Notice that the fact that $\BBB_0$ appears here can be seen from the expression of $d\Psi$, that is
$$
d\Psi_x=\Id_{T_xM_0}+d\psi_x\otimes\nu(x)+\psi(x)d\nu_x\,,
$$
as, by the Gauss--Weingarten relations~\eqref{GW}, $d\nu_x$ is related to $\BBB_0(x)$.\\
Then, by applying the Sobolev inequality holding for $M_0$, we have
\begin{align*}
\Bigl(\int_{M_0}\vert u(x+\psi(x)\nu(x))\vert^{\frac{n}{n-1}}\,d\mu_0(x)\Bigr)^{\frac{n-1}{n}}
\leqslant&\,C_S(M_0,1)\int_{M_0}\bigl\vert \nabla^{0}[u(x+\psi(x)\nu(x))]\bigr\vert\,\,d\mu_0(x)\\
&\,+C_S(M_0,1)\int_{M_0}\vert u(x+\psi(x)\nu(x))\vert\,d\mu_0(x)\\
\leqslant&\,C_S(M_0,1)\int_{M_0}\vert \nabla u(x+\psi(x)\nu(x))\vert\,\vert d\Psi(x)\vert\,d\mu_0(x)\\
&\,+C_S(M_0,1)\int_{M_0}\vert u(x+\psi(x)\nu(x))\vert\,d\mu_0(x)\\
\leqslant&\,C(M_0,\delta)\int_{M}\vert \nabla u(y)\vert\,J\Psi^{-1}(y)\,d\mu(y)\\
&\,+C(M_0,\delta)\int_{M}\vert u(y)\vert\,J\Psi^{-1}(y)\,d\mu(y)\\
\leqslant&\,C(M_0,\delta)\Big(\int_{M}\vert \nabla u(y)\vert\,d\mu(y)+\int_{M}\vert u(y)\vert\,d\mu(y)\Big)\,.
\end{align*}
Hence,
$$
\Bigl(\int_M\vert u\vert^{\frac{n}{n-1}}\,d\mu\Bigr)^{\frac{n-1}{n}}\leqslant 
C(M_0,\delta)\Big(\int_{M}\vert \nabla u\vert\,d\mu+\int_{M}\vert u\vert\,d\mu\Big)\,.
$$
As before, this means that the constant $C(M_0,\delta)$ is a uniform bound on $C_S(M,1)$ for all the hypersurfaces in $\mathfrak{C}^1_\delta(M_0)$, moreover, since $C(M_0,\delta)\to1$, as $\delta\to0$, it also shows the continuous dependence of $C_S(M,1)$ under the $C^1$--convergence of the hypersurfaces.

\begin{thm}
Let $M_0\subseteq\R^{n+1}$ be a smooth, compact hypersurface, embedded in $\R^{n+1}$. Then, there exist uniform bounds, depending only on $M_0$ and $\delta$ (more precisely, on the ``$C^1$-- structure'' of the immersion of $M_0$ in $\R^{n+1}$, its dimension and its second fundamental form), for all the hypersurfaces $M\in\mathfrak{C}^1_\delta(M_0)$ on:
\begin{enumerate}[label={\em({\roman*})}]

\item{the volume of $M$ from above and below away from zero,\label{item0}} 

\item{the Sobolev constants for $p\in[1,n)$ of the embeddings $W^{1,p}(M)\hookrightarrow L^{p^*}(M)$,\label{item1}}

\item{the Sobolev constants for $p\in(n,+\infty]$ of the embeddings $W^{1,p}(M)\hookrightarrow C^{0,1-n/p}(M)$, \label{item2}}

\item{the constants in the Poincar\'e--Wirtinger inequalities on $M$, for $p\in[1,+\infty]$,\label{item3}}

\item{the constants in the embeddings of the fractional Sobolev spaces $W^{s,p}(M)$,\label{item4}}

\item{the constants in the Gagliardo--Nirenberg interpolation inequalities on $M$.\label{item5}}

\end{enumerate}
Moreover, all these bounds go to the corresponding constants for $M_0$, as $\delta\to 0$.\end{thm}
\begin{proof}\ 

{\em\ref{item0}} This is trivial due to the $C^1$--closedness of $M$ to $M_0$.
 
{\em\ref{item1}} As explained at the beginning of the section, we can estimate the constant in the Sobolev inequality for $p\in[1,n)$, by means of $C_S(M,1)$, which is uniformly bounded for all the hypersurfaces $M\in\mathfrak{C}^1_\delta(M_0)$, by the above discussion.

{\em\ref{item2}} If $p>n$, we show that there exists a uniform constant $C(M_0,p,\delta)$ such that
\begin{equation}\label{p>n}
\norma{u}_{C^{0,\alpha}(M)} \leqslant C(M_0,p,\delta)\norma{u}_{W^{1,p}(M)}
\end{equation}
with $\alpha=1 - {n}/{p}$ and 
$$
\norma{u}_{C^{0,\alpha}} = \sup_{y\in M} \vert u(y) \vert + \sup_{y,y^*\in M,\ y\neq y^*}\frac{|u(y)-u(y^*)|}{\vert y- y^* \vert^\alpha}\,,
$$
for all $M\in\mathfrak{C}^1_\delta(M_0)$ and every $C^1$ function $u:M\to\R$.\\
In the same setting and notation at the beginning of this section, it is easy to see that we can choose a special family of hypercubes $Q_i$ such that enlarging their edges of a small value $\sigma>0$, we have hypercubes $\widetilde{Q}_i$ with the further property that $M\cap \widetilde{Q}_i$ can be still written as an orthogonal graph on $\Pi_{p_i}M_0=\R^n\subseteq\R^{n+1}$.\\
The following holds
\begin{align}
\sup_{y\in M} \vert u(y) \vert &\leqslant\sum_{i=1}^k\sup_{y\in M \cap Q_i} \vert u(y)\rho_i(y) \vert
\end{align}
and for every $C^1$ function $v:M\to\R$ with compact support in $M\cap Q_i$, by applying the Sobolev inequality for $p>n$ in $\R^n$ and arguing as in obtaining estimate~\eqref{tangent}, we have
\begin{align}
\sup_{y \in M \cap Q_i} \vert v(y) \vert =&\, \sup_{x\in \R^n}\vert v(x + \theta_i(x)\nu_i)\vert\\
\leqslant&\,C\Big(\int_{\R^n}\bigl\vert \nabla v(x+\theta_i(x)\nu_i)\comp\bigl(\Id+\nabla^{\R^n}\!\theta_i(x)\otimes\nu_i\bigr)\bigr\vert^p\,dx\Bigr)^{1/p}\\
\leqslant&\,C(\delta)\Bigl(\int_{\R^n}\vert \nabla v(x+\theta_i(x)\nu_i)\vert^p\,dx\Bigr)^{1/p}\\
\leqslant&\,C(\delta)\Bigl(\int_{\R^n}\vert \nabla v(x+\theta_i(x)\nu_i)\vert^pJ\Theta\,dx\Bigr)^{1/p}\\
=&\,C(\delta)\Bigl(\int_M\vert \nabla v(y)\vert^p\,d\mu(y)\Bigr)^{1/p}\,,\label{tangent2}
\end{align}
as $J\Theta\geqslant1$.
Setting $v_i=u\rho_i$ and estimating as in getting inequality~\eqref{sobolev1}, we conclude
\begin{equation}
\sup_M \vert u \vert\leqslant C(M_0,p,\delta)\Big(\int_M\vert \nabla u\vert^p+\vert u\vert^p\,d\mu\Big)^{1/p}\,.\label{sup}
\end{equation}
Regarding the seminorm $[u]_{C^{0,\alpha}}= \sup_{y,y^*\in M,\ y\neq y^*}\frac{|u(y)-u(y^*)|}{\vert y- y^* \vert^\alpha}$, given two points $y,y^*\in M$, we have
\begin{align}\label{seminorm}
\vert u(y)-u(y^*) \vert = \Big\vert \sum_{i=1}^{k} v_i(y) - v_i(y^*) \Big\vert \leqslant\sum_{i=1}^k \vert v_i(y) - v_i(y^*) \vert\,.
\end{align}
Then, for any $C^1$ function $v:M\to \R$ with compact support in $M\cap Q_i$, if $y$ and $y^*$ both belong to the intersection of $M$ with the ``enlarged'' hypercube $\widetilde{Q}_i$, we can write $y=x +\theta_i(x)\nu_i$ and $y^*=x^*+\theta_i (x^*)\nu_i$ for some $x,x^*\in\widetilde{Q}_i\cap\Pi_{p_i}M_0$ (by our initial choice of the family $Q_i)$ and there holds
\begin{align}
\vert v(y) - v(y^*) \vert =&\, \vert v(x+\theta_i (x)\nu_i)-v(x^* +\theta_i(x^*)\nu_i) \vert \\
\leqslant &\,C(M_0,p) \,\vert x-x^*\vert ^{\alpha}\,\norma{\nabla^{\R^n}\! (v\comp\Theta)}_{L^p(\R^n)}\\
\leqslant &\,C(M_0,p,\delta) \,\vert y-y^*\vert ^{\alpha}\,\norma{\nabla^{\R^n}\! (v\comp\Theta)}_{L^p(\R^n)}\\
\leqslant &\,C(M_0,p,\delta) \,\vert y - y^* \vert ^{\alpha} \,\norma{\nabla v}_{L^p(M)}\,,\label{ineq_inR}
\end{align}
where the first inequality follows as in the proof of Theorem~4 in Section~5.6.2 of~\cite{Ev}, the second one holds since $\vert x-x^*\vert\leqslant\vert y-y^*\vert$ and the third one is obtained arguing like in estimate~\eqref{tangent2}.\\
If both $y^*$ and $y$ do not belong to $M\cap \widetilde{Q}_i$ clearly $\vert v(y) - v(y^*) \vert=0$, while if $y\in M\cap \widetilde{Q}_i$ with $v(y)\not=0$ but $y^*\not\in M\cap \widetilde{Q}_i$, then $y\in M\cap Q_i$, hence $|y-y^*|\geqslant\sigma$ and
$$
\frac{\vert v(y) - v(y^*) \vert}{|y-y^*|^\alpha}\leqslant\frac{\vert v(y)\vert}{\sigma^\alpha}\leqslant C(M_0,p,\delta)\frac{\norma {\nabla v}_{L^p(M)}}{\sigma^\alpha}\,,
$$
by estimate~\eqref{tangent2}.\\
It follows that, for every $y$ and $y^*$ in $M$, we have
$$
\frac{\vert v(y) - v(y^*) \vert}{|y-y^*|^\alpha}\leqslant C(M_0,p,\delta)(1+\sigma^{-\alpha})\norma {\nabla v}_{L^p(M)}\,.
$$
Then, putting together this and inequality~\eqref{seminorm}, we conclude, for every $y$ and $y^*$ in $M$,
$$
\vert u(y)-u(y^*) \vert \leqslant\sum_{i=1}^{k}\vert v_i(y) - v_i(y^*) \vert \leqslant C(M_0,p,\delta)\, \vert y-y^*\vert^{\alpha}\, \norma{\nabla u}_{W^{1,p}(M)}
$$
which, with inequality~\eqref{sup} gives the desired estimate~\eqref{p>n}.

{\em\ref{item3}} In order to obtain the conclusion for the Poincar\'e--Wirtinger inequality, for any $p\in[1,+\infty]$ and all $M\in\mathfrak{C}^1_\delta(M_0)$,
\begin{equation}
\Vert u - \widetilde{u} \Vert_{L^p(M)} \leqslant C(M_0,p,\delta) \norma{\nabla u}_{L^p(M)}\,,
\end{equation}
where $\widetilde{u}= \fint_M u \, d\mu$, we argue by contradiction assuming this uniform estimate is false. Then, for each $k\in\N$, there would exist a graph hypersurface $M_k\in\mathfrak{C}^1_\delta(M_0)$ and a function $u_k \in W^{1,p}(M_k)$ such that
$$
\| u_k - \widetilde{u}_k\|_{L^{p}(M_k)} \geqslant k \| \nabla u_k \|_{L^{p}(M_k)}.
$$
where $\widetilde{u}_k= \fint_{M_k} u_k \, d\mu_k$. We renormalize these function as
$$
v_k = \frac{u_k - \widetilde{u}_k}{\Vert u_k-\widetilde{u}_k\Vert_{L^p(M_k)}}\,,
$$
then, $\int_{M_k} v_k \, d\mu_k=0$, $\Vert v_k\Vert_{L^p(M_k)}=1$ and $\|\nabla v_k\|_{L^{p}(M_k)}\leqslant1/k$.\\
If we consider the functions $w_k=v_k\comp\Psi_k:M_0\to\R$, where $\Psi_k:M_0\to M_k$ is given by 
$\Psi_k(x)=x+\psi_k(x)\nu(x)$ (as in the second way to deal with $C_S(M,1)$, at the beginning of this section), we have 
\begin{equation}\label{eqcar777}
0<C'(M_0,p,\delta)\leqslant\Vert w_k\Vert_{L^p(M_0)}\leqslant C(M_0,p,\delta)\qquad\text{ and }\qquad\|\nabla w_k\|_{L^{p}(M_0)}\leqslant C(M_0,p,\delta)/k\,.
\end{equation}
In particular, the functions $w_k$ are equibounded in $W^{1,p}(M_0)$, hence by the Rellich--Kondrachov embedding theorem and the second estimate~\eqref{eqcar777}, there exists a subsequence (not relabeled) converging in $L^p(M_0)$ to a constant function equal to some $\lambda\in\R$ which cannot be zero, by the first estimate~\eqref{eqcar777}. Moreover, there holds
$$
\int_{M_0} w_k(x)\,J\Psi_k(x)\,d\mu_0(x)=\int_{M_k} w_k\comp\Psi_k^{-1}(y)\,d\mu_k(y)=\int_{M_k} v_k(y)\,d\mu_k(y)=0\,,
$$
hence, since $J\Psi_k$ are equibounded (formula~\eqref{stimaJPsi}) and assuming, possibly passing again to a subsequence, that $\mathrm{Vol}(M_k)\to V>0$, by means of point (i), we conclude
$$
0=\int_{M_0} (w_k(x)-\lambda)\,J\Psi_k(x)\,d\mu_0(x)+\lambda\int_{M_0}\,J\Psi_k(x)\,d\mu_0(x)\to \lambda V\,,
$$
as $k\to\infty$, being $\int_{M_0}\,J\Psi_k(x)\,d\mu_0(x)=\mathrm{Vol}(M_k)$. This is clearly a contradiction, as $\lambda,V\not=0$ and we are done.\\
The case $p=+\infty$ is analogous.

{\em\ref{item4}} As for the ``usual'' (with integer order) Sobolev spaces, all the constants in the embeddings of the fractional Sobolev spaces are also uniform for the family $\mathfrak{C}^1_\delta(M_0)$. The proof is along the same line, localizing with a partition of unity and using the inequalities holding in $\R^n$ (see~\cite{NePaVa} and~\cite{RuSi}).

{\em\ref{item5}} Finally, we want to show that for any $q,r$ real numbers $1\leqslant q\leqslant+\infty, 1\leqslant r\leqslant+\infty$ and $j,m$ integers $0\leqslant j< m$, there exists a constant $C$ depending on $j,m,r,q,\theta,M_0$ and $\delta$ such that the following interpolation inequalities hold
\begin{equation}\label{interpGN}
\norma{\nabla^j u }_{L^p(M)} \leqslant C \big(\norma{\nabla^m u}_{L^r(M)}+\norma{u}_{L^r(M)}\big)^\theta \norma{u}_{L^q(M)}^{1-\theta},
\end{equation}
for all $M\in\mathfrak{C}^1_\delta(M_0)$, where
\begin{equation}
\frac{1}{p}=\frac{j}{n} + \theta\Bigl(\frac{1}{r}-\frac{m}{n}\Bigr)+\frac{1-\theta}{q}
\end{equation}
for every $\theta \in[{j}/{m},1]$ such that $p$ is nonnegative, with the exception of the case $r=\frac{n}{m-j}\neq1$ for which the inequality is not valid for $\theta=1$.\\
Moreover, if $u:M\to\R$ is a smooth function with $\fint_{M}u\,d\mu=0$, inequality~\eqref{interpGN} simplifies to
\begin{equation}\label{gn2}
\norma{\grad^j u}_{L^p(M)} \leqslant C\norma{\grad^m u}_{L^r(M)}^{\theta} \norma{u}_{L^q(M)}^{1-\theta}\,.
\end{equation}
We can obtain inequality~\eqref{interpGN} arguing as in Proposition~5.1 of~\cite{mant5}, essentially following the line of the proof of Theorem~3.70 in~\cite{aubin0}, but substituting the {\em Sobolev--Poincar\`e inequality}~(41) in the argument there with its version where the constant is uniform for all $M\in\mathfrak{C}^1_\delta(M_0)$.
Indeed, the other ``ingredients'' in such proof are a bound on the volume (uniform, by point~(i)) and some ``universal'' inequalities in which the constants do not depend on the hypersurfaces at 
all~\cite[Theorem~3.69]{aubin0}.\\
Such Sobolev--Poincar\`e inequality~(41) in Theorem~3.70 of~\cite{aubin0} reads
\begin{equation}\label{SP}
\Vert u\Vert_{L^{p^*}(M)}\leqslant C_{SP}(M,p)\Vert \nabla u\Vert_{L^p(M)}\,,
\end{equation}
for every $C^1$--function $u:M\to\R$ (or $u\in W^{1,p}(M)$) with $\int_M u\, d\mu=0$, (here, as before, $p^*=\frac{np}{n-p}$ is the Sobolev conjugate exponent) and we actually need it with a uniform constant, in order to get inequality~\eqref{gn2}, by the very same proof of such theorem.\\
This inequality actually follows by points~(ii) and (iv). Indeed, if $u\in W^{1,p}(M)$, by Sobolev inequality, we have
$$ 
\Vert u\Vert_{L^{p^*}(M)}\leqslant C(M_0,p,\delta)\big(\Vert \nabla u\Vert_{L^p(M)}+\Vert u\Vert_{L^p(M)}\big)
$$
and, by Poincar\`e--Wirtinger inequality, as $\widetilde{u}=\int_M u\, d\mu=0$,
$$
\Vert u\Vert_{L^p(M)}\leqslant C(M_0,p,\delta)\Vert \nabla u\Vert_{L^p(M)}
$$
hence, we obtain inequality~\eqref{SP} with $C_{SP}(M,p)$ bounded by a uniform constant $C(M_0,p,\delta)$, for every $M\in\mathfrak{C}^1_\delta(M_0)$.
\end{proof}

\begin{remark}[The fractional Sobolev spaces $W^{s,p}(M)$]\ \\
At point (v) of the theorem above we considered the fractional Sobolev space $W^{s,p}$ on the hypersurfaces $M\in\mathfrak{C}^1_\delta(M_0)$, which are usually defined via local charts for $M$ and partitions of unity, that is, getting back to the definition with the Gagliardo $W^{s,p}$--seminorms in $\R^n$ (we refer to~\cite{AdamsFournier,Dem,NePaVa,RuSi}, for details). They can be also defined equivalently by considering directly on $M$ the Gagliardo $W^{s,p}$--seminorm of a function $u\in L^p(M)$, for $s\in(0,1)$, as follows:
$$
[u]_{W^{s,p}(M)}^p=\int_{M}\int_{M}\frac{|u(x)-u(y)|^p}{|x-y|^{n+sp}}\,d\mu(x)\,d\mu(y)
$$
and setting $\Vert u\Vert_{W^{s,p}(M)}=\Vert u\Vert_{L^p(M)}+[u]_{W^{s,p}(M)}$. Moreover, the constants giving the equivalence of the two norms obtained by localization or by this direct definition are uniform for all $M\in\mathfrak{C}^1_\delta(M_0)$. Indeed, the localization method of Section~\ref{intro}, is ``uniform'' for all $M\in\mathfrak{C}^1_\delta(M_0)$, meaning that the number of necessary local charts is fixed and the diffeomorphisms between $\R^n$ and ``corresponding'' (associated to correlated local charts, that is, being a graph on the same piece of $M_0$, as in our construction) local ``pieces'' of any different hypersurfaces $M\in\mathfrak{C}^1_\delta(M_0)$, are uniformly ``$C^1$--close'' one to each other.
\end{remark}

\section{Geometric Calder\'on--Zygmund inequalities}\label{CZsec}

\begin{thm}\label{GCZ0}
Let $M_0\subseteq\R^{n+1}$ be a smooth, compact hypersurface, embedded in $\R^{n+1}$ and $p\in(1,+\infty)$. Then, if $\delta>0$ is small enough, there exists a constant $C(M_0,p,\delta)$ such that the following geometric Calder\'on--Zygmund inequality holds,
\begin{equation}\label{GCZ}
\norma{\BBB}_{L^p(M)} \leqslant C(M_0,p,\delta)\big(1+ \norma{\HHH}_{L^p(M)}\big)
\end{equation}
for every $M\in\mathfrak{C}^1_\delta(M_0)$. 
\end{thm}
\begin{proof}
We recall the local graph representation of the hypersurfaces $M \in \mathfrak{C}^1_\delta(M_0)$ over $M_0$, as at the beginning of the previous section. We  have a finite family of hypercubes $Q_i$ centered at $p_i\in M_0$, the partition of unity $\rho_i$ and a parametrization $x\mapsto\Theta(x)=x+\theta_i(x)\nu_i$ on $\Pi_{p_i}M_0\cap Q_i$ of each piece $M\cap Q_i$ of any smooth hypersurface $M\in\mathfrak{C}^1_\delta(M_0)$, where $\nu_i=\nu(p_i)$ and the functions $\theta_i:\Pi_{p_i}M_0\to\R$ satisfy $\Vert\theta_i\Vert_{C^1(\Pi_{p_i}M_0)}\leqslant2\delta$, for every $i\in\{1,\dots,k\}$. Moreover, in dealing with any piece $M\cap Q_i$, we will assume (clearly without losing generality) that $\Pi_{p_i}M_0=\R^n\subseteq\R^{n+1}$ and that $Q_i\cap\Pi_{p_i}M_0$ is the hypercube $Q_{2R}\subseteq\Pi_{p_i}M_0=\R^n$ with edges of length $2R>0$, centered at the origin. Finally, we can also ask that the family of hypercubes $Q'_i\subseteq\R^n$ with edges parallel to the ones of $Q_i$ and of length $R$ (half of the one of $Q_i$), centered at $p_i$, covers any hypersurface $M \in \mathfrak{C}^1_\delta(M_0)$.\\ 
By formulas~\eqref{secondform} and~\eqref{meancurvature}, in the parametrization of $M\cap Q_i$ given by $\Theta$, the second fundamental form $\BBB$ and mean curvature $\HHH$ of $M$ are then expressed by
\begin{equation}\label{secondform+}
\BBB\comp\Theta=-\frac{{\mathrm {Hess}}^{\R^n}\!\theta_i}{\sqrt{1+\vert\nabla^{\R^n}\! \theta_i\vert^2}}\qquad\text{ and }\qquad
\HHH\comp\Theta=-\frac{\Delta^{\R^n}\! \theta_i}{\sqrt{1+\vert\nabla^{\R^n}\! \theta_i\vert^2}}
+\frac{{\mathrm {Hess}}^{\R^n}\!\theta_i(\nabla^{\R^n}\! \theta_i,\nabla^{\R^n}\! \theta_i)}{\big(\sqrt{1+\vert\nabla^{\R^n}\! \theta_i\vert^2}\big)^3}\,.
\end{equation}
Letting  and $\rho:\R^{n}\to[0,1]$ a cut--off function with compact support in $Q_{2R}$ and equal to 1 on $Q_R=Q'_i\cap \Pi_{p_i}M_0$ and setting $A_R=\{(x,\theta_i(x))\ :\ x\in Q_R\}$, $A_{2R}=\{(x,\theta_i(x))\ :\ x\in Q_{2R}\}$, we have
\begin{equation}\label{q1}
\Vert\BBB\Vert_{L^p(A_R)}^p=\int_{Q_R}|\BBB\comp\Theta|^p J\Theta\,dx
\leqslant\int_{Q_R}\rho^p|{\mathrm {Hess}}^{\R^n}\!\theta_i|^p\,dx=\int_{\R^n}|\rho{\mathrm {Hess}}^{\R^n}\!\theta_i|^p\,dx\,,
\end{equation}
as $\mu=J\Theta\,\mathscr{L}^n$ and $J\Theta=\sqrt{1+\vert\nabla^{\R^n}\!\theta_i\vert^2\,}$. Then, we estimate
\begin{align}
\int_{\R^n}|\rho{\mathrm {Hess}}^{\R^n}\!\theta_i|^p\,dx\leqslant&\,C \int_{\R^n}|{\mathrm {Hess}}^{\R^n}\!(\rho \theta_i)|^p\,dx
+C \int_{\R^n}|2\nabla^{\R^n}\!\rho\otimes\nabla^{\R^n}\! \theta_i|^p\,dx
+C \int_{\R^n}|\theta_i{\mathrm {Hess}}^{\R^n}\!\rho|^p\,dx\\
\leqslant&\,C \int_{\R^n}|{\mathrm {Hess}}^{\R^n}\!(\rho \theta_i)|^p\,dx+C\,,
\end{align}
where $C=C(M_0,p,\delta)$, as the last two integrals in the first line are clearly bounded by a constant $C=C(M_0,p,\delta)$.\\
Hence, by applying the standard Calder\'on--Zygmund estimates in $\R^n$ (see~\cite{GilbTrud}, for instance) to the last term above, we get
\begin{align}
\int_{\R^n}|\rho{\mathrm {Hess}}^{\R^n}\!\theta_i|^p\,dx\leqslant&\,C\int_{\R^n}|\Delta^{\R^n}\!(\rho \theta_i)|^p\,dx+C\\
\leqslant&\,C\int_{\R^n}|\rho\Delta^{\R^n}\! \theta_i|^p\,dx+C\int_{\R^n}|2\langle\nabla^{\R^n}\!\rho\,|\,\nabla^{\R^n}\! \theta_i\rangle|^p\,dx+C\int_{\R^n}|\theta_i\Delta^{\R^n}\!\rho|^p\,dx\\
\leqslant&\,C\int_{\R^n}\Bigl\vert-\rho(\HHH\comp\Theta)\sqrt{1+\vert\nabla^{\R^n}\! \theta_i\vert^2}+\frac{\rho{\mathrm {Hess}}^{\R^n}\!\theta_i(\nabla^{\R^n}\! \theta_i,\nabla^{\R^n}\! \theta_i)}{1+\vert\nabla^{\R^n}\! \theta_i\vert^2}\,\Bigl\vert^p\,dx+C\\
\leqslant&\,C\int_{\R^n}\vert\rho(\HHH\comp\Theta)\vert^p\,dx+C\int_{\R^n}\vert\rho{{\mathrm {Hess}}^{\R^n}\!\theta_i(\nabla^{\R^n}\! \theta_i,\nabla^{\R^n}\! \theta_i)}\vert^p\,dx+C\\
\leqslant&\,C\int_{\R^n}\vert\rho(\HHH\comp\Theta)\vert^p\,dx+C\int_{\R^n}|\nabla^{\R^n}\! \theta_i|^{2p}\vert\rho{\mathrm {Hess}}^{\R^n}\!\theta_i\vert^p\,dx+C
\end{align}
where the constant $C$ depends only on $M_0$, $p$ and $\delta$ (we estimated the last two integrals in the second line with such a constant, as we did above for the Hessian).\\
If $\delta>0$ is small enough, then $C|\nabla^{\R^n}\! \theta_i|^{2p}<1/2$ and we get
$$
\int_{\R^n}|\rho{\mathrm {Hess}}^{\R^n}\!\theta_i|^p\,dx\leqslant 2C\int_{\R^n}\vert\rho(\HHH\comp\Theta)\vert^p\,dx+2C\leqslant 2C\int_{Q_{2R}}\vert(\HHH\comp\Theta)\vert^p\,dx+2C
$$
which clearly implies, by formula~\eqref{q1},
\begin{equation}
\Vert\BBB\Vert_{L^p(A_R)}\leqslant C\int_{Q_{2R}}\vert(\HHH\comp\Theta)\vert^p\,dx+C\leqslant C\int_{Q_{2R}}\vert(\HHH\comp\Theta)\vert^pJ\Theta\,dx+C\leqslant C\big(1+\Vert\HHH\Vert_{L^p(A_{2R})}^p\big)\,,
\end{equation}
with $C=C(M_0,p,\delta)$.\\
Hence, by construction and invariance by isometry,
\begin{equation}
\Vert\BBB\Vert_{L^p(M\cap Q_i')}\leqslant C\big(1+\Vert\HHH\Vert_{L^p(M\cap Q_i)}^p\big)\leqslant C\big(1+\Vert\HHH\Vert_{L^p(M)}^p\big)\,.
\end{equation}
Since the number of hypercubes $Q_i'$ covering $M$ is fixed and $C=C(M_0,p,\delta)$, we obtain the thesis of the theorem.
\end{proof}

We have an analogous theorem for Schauder estimates, after defining appropriately the H\"older $C^{0,\alpha}$--norm of a tensor $T$ on $M$, that is,
$$
\Vert T\Vert_{C^{0,\alpha}(M)}=\sup_M\vert T\vert+[T]_{C^{0,\alpha}(M)}
$$
where we need to give a meaning to the seminorm $[T]_{C^{0,\alpha}(M)}$.\\

If $T$ is an $m$--form (hence, a covariant $m$--tensor), one possibility is to ``extend the action'' of the tensor $T$ from the bundle $\oplus^mTM$ of covariant $m$-–tensors on $M$ to the one of the whole ``ambient'' $\R^{n+1}$ by means of the orthogonal projection on the tangent bundle $TM$ (as we identify $T_xM$ with a vector subspace of $T_x \R^{n+1} \approx \R^{n+1}$, for every $x \in M$).
To give an example, if $T=\BBB$, letting $\pi_x:\R^{n+1}\to T_xM$ be the orthogonal projection on the tangent space of $M$, for every $x\in M$, we can define the ``extension'' of $\BBB$ (without relabeling it) by considering at every $x\in M$ the bilinear form $\BBB_x:\oplus^2T_x\R^{n+1}\approx\R^{n+1}\times \R^{n+1}\to\R$ as $\BBB_x(v,w)=\BBB_x(\pi_x(v),\pi_x(w))$. Extending analogously a general $m$--form $T$ from operating on $\oplus^mTM$ to $\oplus^mT\R^{n+1}$, its norm as a multilinear functional is unchanged at every point $x\in M$ and we can then consider its components $T_{j_1\dots j_m}$ in the canonical basis of $\R^{n+1}$ to define
$$
[T]_{C^{0,\alpha}(M)}=\sum_{j_1,\dots,j_m=1}^{n+1}[T_{j_1\dots j_m}]_{C^{0,\alpha}(M)}=\sum_{j_1,\dots,j_m=1}^{n+1}\sup_{{\genfrac{}{}{0pt}{}
{x,y\in M}{x\neq y}}}\frac{|T_{j_1\dots j_m}(x)-T_{j_1\dots j_m}(y)|}{\vert x- y \vert^\alpha}\,.
$$
Finally, if the tensor is of general type (it has also contravariant components), we ``transform'' it in a covariant one by means of the musical isomorphisms (see~\cite{gahula}, for instance) and then proceed as above. Anyway, in the following all the tensors will be covariant.

\begin{remark}\label{holder0}
This ``global'', partially coordinate--free definition (only the canonical coordinates of $\R^{n+1}$ are involved, not any coordinate chart for $M$) is useful in general, but in our special case of families of hypersurfaces which are representable as graphs on a fixed one, we can also consider an {\em equivalent} H\"older seminorm by means of the local description of $M$ with the hypercubes $Q_i$, which is more convenient for our computations. For any $m$--form $T$ on $M$, we set (in the notation of the proof of Theorem~\ref{GCZ0})
\begin{align}
[T]_{C^{0,\alpha}(V)}=&\,\sum_{j_1,\dots,j_m=1}^{n}[T_{j_1\dots j_m}\comp\Theta]_{C^{0,\alpha}(\Theta^{-1}(V))}\\
=&\,\sum_{j_1,\dots,j_m=1}^{n}\sup_{{\genfrac{}{}{0pt}{}
{x,y\in \Theta^{-1}(V)}{x\neq y}}}\frac{|T_{j_1\dots j_m}(\Theta(x))-T_{j_1\dots j_m}(\Theta(y))|}{\vert x- y \vert^\alpha}\,,\label{holder}
\end{align}
for every open set $V\subseteq M\cap Q_i$, where $T_{j_1\dots j_m}$ are the components of $T$ in the parametrization $x\mapsto\Theta(x)=x+\theta_i(x)e_{n+1}$. Then, we define
$$
[T]_{C^{0,\alpha}(M)}=\sum_{i=1}^k[T]_{C^{0,\alpha}(A_R)},
$$
by means of the finite family of sets $A_R$ (whose number is fixed) covering $M\in\mathfrak{C}^{1}_\delta(M_0)$. 
\end{remark}

\begin{thm}\label{SCH0}
Let $M_0\subseteq\R^{n+1}$ be a smooth, compact hypersurface, embedded in $\R^{n+1}$ and $\alpha\in(0,1]$. Then, if $\delta>0$ is small enough, there exists a constant $C(M_0,\alpha,\delta)$ such that the following geometric Schauder estimate holds,
\begin{equation}\label{SCH}
\Vert\BBB\Vert_{C^{0,\alpha}(M)}\leqslant C(M_0,\alpha,\delta)\,\big(1+\Vert \HHH\Vert_{C^{0,\alpha}(M)}\big)
\end{equation}
for every $M\in\mathfrak{C}^{1,\alpha}_\delta(M_0)$. 
\end{thm}
\begin{proof}
In the same setting and notation of the proof of Theorem~\ref{GCZ0}, for every hypercube $Q_i$, the function $\theta_i$ belongs to $C^{1,\alpha}(Q_{2R})$, with $\Vert{\theta_i}\Vert_{C^{1,\alpha}(Q_{2R})}\leqslant2\delta$. Then, keeping into account Remark~\ref{holder0}, we deal with $\Vert\BBB\Vert_{C^{0,\alpha}(A_R)}$, which satisfies
\begin{equation}\label{q2}
\Vert\BBB\Vert_{C^{0,\alpha}(A_R)}=\Vert\BBB\comp\Theta\Vert_{C^{0,\alpha}(Q_R)}=\biggl\Vert\frac{{\mathrm{Hess}}^{\R^n}\!\theta_i}{\sqrt{1+\vert\nabla^{\R^n}\! \theta_i\vert^2}}\biggr\Vert_{C^{0,\alpha}(Q_R)}\leqslant C\,\Vert \theta_i\Vert_{C^{2,\alpha}(Q_R)}\,,
\end{equation}
by equality~\eqref{secondform+} and since $Q_R=\Theta^{-1}(A_R)$, by construction.\\
Hence, by the standard Schauder estimates in $Q_{2R}=\Theta^{-1}(A_{2R})$ (see~\cite{GilbTrud}, for instance), we get
\begin{align}
&\Vert{\theta_i}\Vert_{C^{2,\alpha}(Q_R)}\\
&\leqslant C\,\Vert\Delta^{\R^n}\! \theta_i\Vert_{C^{0,\alpha}(Q_{2R})}+C\Vert{\theta_i}\Vert_{C^{1,\alpha}(Q_{2R})}\\
&\leqslant C\,\biggl\Vert -(\HHH\comp\Theta)\sqrt{1+\vert\nabla^{\R^n}\! \theta_i\vert^2}+\frac{{\mathrm {Hess}}^{\R^n}\!\theta_i(\nabla^{\R^n}\! \theta_i,\nabla^{\R^n}\! \theta_i)}{1+\vert\nabla^{\R^n}\! \theta_i\vert^2}\,\biggl\Vert_{C^{0,\alpha}(Q_{2R})}+\,C\\
&\leqslant C\,\Vert \HHH\comp\Theta\Vert_{C^{0,\alpha}(Q_{2R})}+C\,\Vert\nabla^{\R^n}\! \theta_i\Vert^2_{C^{0,\alpha}(Q_{2R})}\Vert{\mathrm {Hess}}^{\R^n}\!\theta_i\Vert_{C^{0,\alpha}(Q_{2R})}+C\\
&\leqslant C\,\Vert \HHH\comp\Theta\Vert_{C^{0,\alpha}(Q_{2R})}+C\delta^2\Vert \theta_i\Vert_{C^{2,\alpha}(Q_{2R})}+C\,,
\end{align}
where the constant $C$ depends only on $M_0$, $\alpha$ and $\delta$, as $\Vert{\theta_i}\Vert_{C^{1,\alpha}(Q_{2R})}\leqslant2\delta$. This estimate clearly implies, by formula~\eqref{q2} and equality~\eqref{secondform+},
\begin{equation}
\Vert\BBB\Vert_{C^{0,\alpha}(A_R)}\leqslant C\,\Vert \HHH\Vert_{C^{0,\alpha}(M)}+C\delta^2\Vert \BBB\Vert_{C^{0,\alpha}(M)}+C
\end{equation}
and since the family of sets $A_R$ covering $M\in\mathfrak{C}^{1}_\delta(M_0)$ is finite and its number is fixed, we conclude
\begin{equation}
\Vert\BBB\Vert_{C^{0,\alpha}(M)}\leqslant C\,\Vert \HHH\Vert_{C^{0,\alpha}(M)}+C\delta^2\Vert \BBB\Vert_{C^{0,\alpha}(M)}+C\,,
\end{equation}
with a constant $C$ depending only on $M_0$, $\alpha$ and $\delta$ (and we can clearly choose $C$ to be monotonically increasing with $\delta$).\\
Then, if $\delta>0$ is small enough, we have $C\delta^2\,\Vert\BBB\Vert^2_{C^{0,\alpha}(M)}<\Vert\BBB\Vert^2_{C^{0,\alpha}(M)}/2$, hence we get
\begin{equation}
\Vert\BBB\Vert_{C^{0,\alpha}(M)}\leqslant 2C\,\Vert \HHH\Vert_{C^{0,\alpha}(M)}+2C\,,
\end{equation}
that is,
\begin{equation}
\Vert\BBB\Vert_{C^{0,\alpha}(M)}\leqslant C\,\big(1+\Vert \HHH\Vert_{C^{0,\alpha}(M)}\big)\,,
\end{equation}
where the constant $C$ depends only on $M_0$, $\alpha$ and $\delta$, which is the thesis of the theorem.
\end{proof}

We now deal with families of $n$--dimensional graph hypersurfaces in $M\in\mathfrak{C}^1_\delta(M_0)$ over $M_0$ with a uniform bound $\Vert\BBB\Vert_{L^\infty(M)}$ on the second fundamental form.

Arguing again in the same setting and notation of the proof of Theorem~\ref{GCZ0}, for $p\in(1,+\infty)$ and any $C^2$--function $u:M\to\R$ (or $u\in W^{2,p}(M)$), we have
\begin{equation}
\norma{\nabla^2 u}_{L^p(M)} \leqslant C\sum_{i=1}^k \norma{\nabla^2 (u \rho_i)}_{L^p(M \cap Q_i)}\label{cz1bis}
\end{equation}
(here $\nabla$ is the Levi--Civita connection of $M$) and, for every $C^2$ function $v:M\to\R$, with compact support in $M\cap Q_i$, there holds
\begin{align}
\int_{M\cap Q_i}\vert \nabla^2v(y)\vert^p\,d\mu(y)
=&\,\int_{\R^n}\big\vert (\nabla^2v)(x+\theta_i(x)\nu_i)\big\vert^p J\Theta(x)\,dx\\
\leqslant&\,C(\delta)\int_{\R^n}\big\vert (\nabla^2v)(x+\theta_i(x)\nu_i)\big\vert^p\,dx\,,\label{eqcar999}
\end{align}
as $J\Theta=\sqrt{1+|\nabla^{\R^n}\!\theta_i|^2\,}\leqslant1+2\delta$.\\
In the coordinates given by the parametrization $\Theta$, the coefficients of the metric $g$ of $M$ (induced by $\R^{n+1}$) in $M\cap Q_i$ are 
$$
g_{\ell m}(\Theta(x))=\delta_{\ell m}+\frac{\pa \theta_i}{\pa x_\ell}(x)\frac{\pa \theta_i}{\pa x_m}(x)\,,
$$
hence, they and the ones of the inverse matrix are bounded by a constant depending only on $M_0$ and $\delta$. By formula~\eqref{Chr}, the Christoffel symbols  of the Levi--Civita connection $\nabla$ satisfy
\begin{equation}\label{eqcar888}
|\Gamma^{s}_{\ell m}(\Theta(x))|\leqslant C\sum_{p,q,r=1}^n\Big|\frac{\partial (g_{pq}\comp\Theta)}{\partial x_r}(x)\Big|=C\sum_{p,q,r=1}^n\Big|\frac{\pa^2 \theta_i}{\pa x_r\pa x_p}(x)\frac{\pa \theta_i}{\pa x_q}(x)\Big|\,.
\end{equation}
Then, recalling the first formula~\eqref{GW}, 
\begin{align}
\Big\vert\frac{\pa^2\theta_i}{\pa x_\ell\pa x_m}(x)\Big\vert=&\,\Big\vert\frac{\pa^2\Theta}{\pa x_\ell\pa x_m}(x)\Big\vert\\
=&\,\Big\vert\Gamma_{\ell m}^s(\Theta(x))\frac{\pa\Theta}{\pa x_s}(x)-\BBB_{\ell m}(\Theta(x))\nu(\Theta(x))\Big\vert\\
\leqslant&\,C\vert\Gamma_{\ell m}^s(\Theta(x))\vert\,\Big\vert\frac{\pa\Theta}{\pa x_s}(x)\Big\vert+|\BBB_{\ell m}(\Theta(x))|\\
\leqslant&\,C\vert\mathrm{Hess}^{\R^n}\!\theta_i(x)\vert\,\vert\nabla^{\R^n}\!\theta_i(x)\vert\,(1+\vert\nabla^{\R^n}\!\theta_i(x)\vert)+|\BBB(\Theta(x))|\,,\label{HessBBB}
\end{align}
where in the last passage we estimated the Christoffel symbols by means of inequality~\eqref{eqcar888}. As $\vert\nabla^{\R^n}\!\theta_i\vert\leqslant2\delta$, we conclude
$$
|\mathrm{Hess}^{\R^n}\!\theta_i(x)|\leqslant C|\mathrm{Hess}^{\R^n}\!\theta_i(x)|\,|\nabla^{\R^n}\!\theta_i(x)|+C\vert\BBB(\Theta(x))\vert\leqslant C|\mathrm{Hess}^{\R^n}\!\theta_i(x)|\delta+C\vert\BBB(\Theta(x))\vert
$$
with a constant $C$ depending only on $\delta$, which implies, if $\delta$ is smaller than $1/2C$, the estimate
\begin{equation}\label{HB}
|\mathrm{Hess}^{\R^n}\!\theta_i(x)|\leqslant 2C(M_0,\delta)\vert\BBB(\Theta(x))\vert\,,
\end{equation}
for every $x\in Q_i\cap\Pi_{p_i}M\subseteq\R^n$.\\
By the first formula~\eqref{eqcar888}, it follows
\begin{equation}\label{GammaB}
|\Gamma_{\ell m}^s(\Theta(x))|\leqslant C|\mathrm{Hess}^{\R^n}\!\theta_i(x)|\,\vert\nabla^{\R^n}\!\theta_i\vert\leqslant C\delta\vert\BBB(\Theta(x))\vert
\end{equation}
with $C=C(\delta)$, then computing schematically, we have
\begin{equation}\label{hesseq}
(\nabla^2v)(\Theta(x))={\mathrm{Hess}}^{\R^n}\!(v\comp\Theta)(x)-\Gamma(\Theta(x))\star\nabla^{\R^n}\!(v\comp\Theta)(x)\,,
\end{equation}
hence,
$$
|(\nabla^2v)(\Theta(x))|\leqslant C|{\mathrm{Hess}}^{\R^n}(v\comp\Theta)(x)|+C\delta\vert\BBB(\Theta(x))\vert\,\vert\nabla^{\R^n}(v\comp\Theta)(x)\vert\,.
$$
Applying the Calder\'on--Zygmund inequality in $\R^n$, we get
\begin{align}
\int_{\R^n}\big\vert (\nabla^2v)(x+\theta_i(x)\nu_i)\big\vert^p\,dx
\leqslant&\,C\int_{\R^n}|{\mathrm{Hess}}^{\R^n}[v(x+\theta_i(x)\nu_i)]|^p\,dx\\
&\,+C\delta\int_{\R^n}\vert\BBB(\Theta(x))\vert^p\,\vert\nabla^{\R^n} [v(x+\theta_i(x)\nu_i])\vert^p\,dx\\
\leqslant&\,C\int_{\R^n}|\Delta^{\R^n}[v(x+\theta_i(x)\nu_i)]|^p\,dx\\
&\,+C(\delta)\int_{\R^n}\vert\BBB(\Theta(x))\vert^p\,\vert\nabla v(\Theta(x))\vert^p\,dx\,.\\
\leqslant&\,C\int_{\R^n}|\Delta^{\R^n}[v(x+\theta_i(x)\nu_i)]|^p\,dx\\
&\,+C(\delta)\int_{M\cap Q_i}\vert\BBB(y)\vert^p\vert\nabla v(y)\vert^p\,d\mu(y)\,,\label{tangent3}
\end{align}
arguing as in estimate~\eqref{tangent2} to get the last inequality.\\
Contracting equation~\eqref{hesseq} with the inverse of the metric and estimating, we have 
$$
|\Delta^{\R^n}(v\comp\Theta)(x)|\leqslant C|(\Delta v)(\Theta(x))|+C\delta|(\BBB\comp\Theta)(x)|\,|\nabla^{\R^n}(v\comp\Theta)(x)|\,
$$
thus, by inequalities~\eqref{eqcar999} and~\eqref{tangent3}, we obtain
\begin{align}
\int_{M\cap Q_i}\vert \nabla^2v(y)\vert^p\,d\mu(y)
\leqslant&\,C\int_{\R^n}|(\Delta v)(x+\theta_i(x)\nu_i)|^p\,dx+C\int_{M\cap Q_i}\vert\BBB(y)\vert^p\vert\nabla v(y)\vert^p\,d\mu(y)\\
\leqslant&\,C\int_{M\cap Q_i}|\Delta v(y)|^p\,d\mu(y)+C\int_{M\cap Q_i}\vert\BBB(y)\vert^p\vert\nabla v(y)\vert^p\,d\mu(y)\,,
\end{align}
with $C=C(M_0,\p,\delta)$, arguing again as above.\\
Getting back to inequality~\eqref{cz1bis}, we conclude
\begin{align}
\norma{\nabla^2 u}_{L^p(M)}^p \leqslant &\,C\sum_{i=1}^k \norma{\nabla^2 (u \rho_i)}_{L^p(M \cap Q_i)}^p\\
\leqslant &\,C\sum_{i=1}^k \int_{M\cap Q_i}|\Delta (u\rho_i)|^p\,d\mu+C\int_{M\cap Q_i}\vert\BBB\vert^p\vert\nabla (u\rho_i)\vert^p\,d\mu\\
\leqslant &\,C\sum_{i=1}^k \int_{M\cap Q_i}|\Delta u|^p\,d\mu+C\int_{M\cap Q_i}\big(|u|^p+\vert\nabla u\vert^p\big)\,d\mu\\
\leqslant &\,C\int_{M}|\Delta u|^p\,d\mu+C\int_{M}\big(|u|^p+\vert\nabla u\vert^p\big)\,d\mu\,,\label{Bmod}
\end{align}
with $C=C(M_0,p,\delta,\Vert\BBB\Vert_{L^\infty(M)})$.
Interpolating the integral of $\vert\nabla u\vert^p$ between $\norma{\nabla^2 u}_{L^p(M)}$ and $\norma{u}_{L^p(M)}$ by means of the uniform Gagliardo--Nirenberg inequalities of the previous section, we obtain the following theorem.

\begin{thm}\label{GCZ1}
Let $M_0\subseteq\R^{n+1}$ be a smooth, compact hypersurface, embedded in $\R^{n+1}$ and $p\in(1,+\infty)$. Then, if $\delta>0$ is small enough, there exists a constant $C$ which depends only on $M_0$, $p$, $\delta$ and $\Vert\BBB\Vert_{L^\infty(M)}$ such that the following Calder\'on--Zygmund inequality holds,
\begin{equation}\label{c1}
\norma{\nabla^2 u}_{L^p(M)}\leqslant C\norma{\Delta u}_{L^p(M)}+C \norma{u}_{L^p(M)}
\end{equation}
hence,
\begin{equation}\label{c2}
\norma{u}_{W^{2,p}(M)}\leqslant C\norma{\Delta u}_{L^p(M)}+C \norma{u}_{L^p(M)}\,,
\end{equation}
for every hypersurface $M\in\mathfrak{C}^{1}_\delta(M_0)$ and $u\in W^{2,p}(M)$.
\end{thm}

\begin{remark}\label{GCZ1bis}
Notice that if $p<n$, we can modify the chain of inequalities~\eqref{Bmod} as follows,
\begin{align}
\norma{\nabla^2 u}_{L^p(M)}^p \leqslant &\,C\sum_{i=1}^k \norma{\nabla^2 (u \rho_i)}_{L^p(M \cap Q_i)}^p\\
\leqslant &\,C\sum_{i=1}^k \int_{M\cap Q_i}|\Delta (u\rho_i)|^p\,d\mu+C\int_{M\cap Q_i}\vert\BBB\vert^p\vert\nabla (u\rho_i)\vert^p\,d\mu\\
\leqslant &\,C\sum_{i=1}^k \int_{M\cap Q_i}|\Delta (u\rho_i)|^p\,d\mu\\
&\,+C\Bigl(\int_{M\cap Q_i}\vert\BBB\vert^n\,d\mu\Bigr)^{p/n}\Bigl(\int_{M\cap Q_i}
\vert\nabla (u\rho_i)\vert^{np/(n-p)}\,d\mu\Bigr)^{(n-p)/n}\\
\leqslant &\,C\sum_{i=1}^k \int_{M\cap Q_i}|\Delta (u\rho_i)|^p\,d\mu+C\Vert\BBB\Vert_{L^n(M\cap Q_i)}^p\norma{\nabla^2 (u \rho_i)}_{L^p(M \cap Q_i)}^p\,.
\end{align}
Hence, arguing as before, it is easy to conclude that inequalities~\eqref{c1} and~\eqref{c2} hold with a constant $C=C(M_0,p,\delta,\Vert\BBB\Vert_{L^n(M)})$, if $\delta>0$ is small enough. Moreover, since we have seen in Theorem~\ref{GCZ0} that a control on $\Vert\HHH\Vert_{L^n(M)}$ implies a control on $\Vert\BBB\Vert_{L^n(M)}$, we have uniform Calder\'on--Zygmund inequalities for families of $n$--dimensional graph hypersurfaces over $M_0$, with mean curvature uniformly bounded  in $L^n(M)$.
\end{remark}

With a similar argument, computing as in Theorem~\ref{SCH0}, we have analogous Schauder estimates for $C^{2,\alpha}$ functions $u:M\to\R$, with $M\in\mathfrak{C}^{1,\alpha}_\delta(M_0)$ and $\delta>0$ small enough, 
\begin{equation}\label{cc2}
\norma{u}_{C^{2,\alpha}(M)}\leqslant C\norma{\Delta u}_{C^{0,\alpha}(M)}+C \norma{u}_{C^{0,\alpha}(M)}\,,
\end{equation}
where the constant $C$ depends only on $M_0$, $\alpha\in(0,1]$, $\delta$ and $\Vert\BBB\Vert_{C^{0,\alpha}(M)}$ (or $\Vert\HHH\Vert_{C^{0,\alpha}(M)}$, by Theorem~\ref{SCH0}).

\begin{remark}\label{tenrem}
Localizing and computing in coordinates (see Remark~\ref{holder0}), it is easy to generalize estimates~\eqref{c1},~\eqref{c2} and~\eqref{cc2} also to tensors, under the same hypotheses. The same holds also for all the estimates of the previous section (see~\cite{mant5} for an example of how this can be done).
\end{remark}

\subsection{Geometric higher order Calder\'on--Zygmund estimates}\ 

We let $M_0$ as above and $p>1$, we want now to deal with $\norma{\nabla^k\BBB}_{L^p(M)}$, assuming that we have a uniform bound on $\Vert\HHH\Vert_{L^q(M)}$ with $q>n$, where $M$ is any $n$--dimensional graph hypersurface over $M_0$ in $\mathfrak{C}^1_\delta(M_0)$, if $\delta>0$ is small enough.

\begin{thm}\label{GCZ2}
Let $M_0\subseteq\R^{n+1}$ be a smooth, compact hypersurface, embedded in $\R^{n+1}$. Then, for any $q>n$, if $\delta>0$ is small enough, there exists a constant $C$ which depends only on $M_0$, $p$, $q$, $\delta$ and $\Vert\HHH\Vert_{L^q(M)}$, such that the following geometric higher order Calder\'on--Zygmund inequality holds, for $p\in(1,n)$,
$$
\norma{\nabla^k\BBB}_{L^p(M)}\leqslant C\big(1+\norma{\nabla^k\HHH}_{L^p(M)}\big)
$$
hence,
\begin{equation}\label{CZG2}
\norma{\BBB}_{W^{k,p}(M)} \leqslant C\big(1+ \norma{\HHH}_{W^{k,p}(M)}\big)\,,
\end{equation}
for any hypersurface $M\in \mathfrak{C}^1_\delta(M_0)$ and $k\in\N$.\\
Moreover, the same inequalities hold for any $p\in(1,+\infty)$ with a constant $C$ depending only on $M_0$, $p$, $\delta$ and $\Vert\BBB\Vert_{L^\infty(M)}$.
\end{thm}
\begin{proof}
We first deal with the case $p\in(1,n)$. Fixed $k\in\N$, by means of inequality~\eqref{c1}, which holds with a constant $C=C(M_0,p,\delta,\Vert\BBB\Vert_{L^n(M)})$, by Remark~\ref{GCZ1bis} and taking into account Remark~\ref{tenrem}, we have
\begin{align}
\norma{\nabla^k\BBB}_{L^p(M)}=&\,\norma{\nabla_{i_1}\nabla_{i_2}(\nabla_{i_3}\cdots\nabla_{i_k}\BBB)}_{L^p(M)}\\
\leqslant &\,C\norma{\Delta(\nabla_{i_3}\cdots\nabla_{i_k}\BBB)}_{L^p(M)}+C\norma{\nabla_{i_3}\cdots\nabla_{i_k}\BBB}_{L^p(M)}\\
=&\,C\norma{g^{\ell m}\nabla_\ell\nabla_m\nabla_{i_3}\cdots\nabla_{i_k}\BBB}_{L^p(M)}+C\norma{\nabla^{k-2}\BBB}_{L^p(M)}\\
\leqslant&\,C\norma{g^{\ell m}\nabla_\ell\nabla_{i_3}\nabla_{m}\cdots\nabla_{i_k}\BBB}_{L^p(M)}+C\norma{\nabla^{k-2}\BBB}_{L^p(M)}\\
&\,+C\norma{\Riem\star\nabla^{k-2}\BBB}_{L^p(M)}+C\norma{\nabla\Riem\star\nabla^{k-3}\BBB}_{L^p(M)}\\
\leqslant&\,C\norma{g^{\ell m}\nabla_\ell\nabla_{i_3}\nabla_{i_4}\nabla_{m}\cdots\nabla_{i_k}\BBB}_{L^p(M)}+C\norma{\nabla^{k-2}\BBB}_{L^p(M)}\\
&\,+C\norma{\Riem\star\nabla^{k-2}\BBB}_{L^p(M)}+C\norma{\nabla\Riem\star\nabla^{k-3}\BBB}_{L^p(M)}\\
&\,+C\norma{\nabla^2\Riem\star\nabla^{k-4}\BBB}_{L^p(M)}\\
&\,\qquad\qquad\qquad\qquad\qquad{\cdots}\\
\leqslant&\,C\norma{g^{\ell m}\nabla_\ell\nabla_{i_3}\nabla_{i_4}\cdots\nabla_{i_k}\nabla_m\BBB}_{L^p(M)}+C\norma{\nabla^{k-2}\BBB}_{L^p(M)}\\
&\,+C\sum_{s=0}^{k-2}\norma{\nabla^s\Riem\star\nabla^{k-2-s}\BBB}_{L^p(M)}\\
\leqslant&\,C\norma{g^{\ell m}\nabla_{i_3}\nabla_\ell\nabla_{i_4}\cdots\nabla_{i_k}\nabla_m\BBB}_{L^p(M)}+C\norma{\nabla^{k-2}\BBB}_{L^p(M)}\\
&\,+C\sum_{s=0}^{k-2}\norma{\nabla^s\Riem\star\nabla^{k-2-s}\BBB}_{L^p(M)}\\
&\,\qquad\qquad\qquad\qquad\qquad{\cdots}\\
\leqslant&\,C\norma{g^{\ell m}\nabla_{i_3}\nabla_{i_4}\cdots\nabla_{i_k}\nabla_\ell\nabla_m\BBB}_{L^p(M)}+C\norma{\nabla^{k-2}\BBB}_{L^p(M)}\\
&\,+C\sum_{s=0}^{k-2}\norma{\nabla^s\Riem\star\nabla^{k-2-s}\BBB}_{L^p(M)}\\
=&\,C\norma{\nabla^{k-2}\Delta\BBB}_{L^p(M)}+C\norma{\nabla^{k-2}\BBB}_{L^p(M)}\\
&\,+C\sum_{s=0}^{k-2}\norma{\nabla^s\Riem\star\nabla^{k-2-s}\BBB}_{L^p(M)}
\end{align}
where the symbol $T \star S$ (following Hamilton~\cite{hamilton1}) denotes a tensor formed by a sum of terms each one given by some contraction of the pair $T$, $S$ with the inverse of the metric $g^{ij}$. A very useful property of such $\star$ product is that
$|T\star S|\leqslant C|T||S|$ where the constant $C$ depends only on the ``algebraic structure'' of $T\star S$, moreover, it clearly holds $\nabla T\star S=\nabla T\star S+T\star\nabla S$.\\
By formula~\eqref{Gauss-eq} for the Riemann tensor, we can then write $\Riem=\BBB\star\BBB$, hence
\begin{align}
\norma{\nabla^k\BBB}_{L^p(M)}
\leqslant&\,C\norma{\nabla^{k-2}\Delta\BBB}_{L^p(M)}+C\norma{\nabla^{k-2}\BBB}_{L^p(M)}+C\sum_{s=0}^{k-2}\norma{\nabla^s(\BBB\star\BBB)\star\nabla^{k-2-s}\BBB}_{L^p(M)}\\
\leqslant&\,C\norma{\nabla^{k-2}\Delta\BBB}_{L^p(M)}+C\norma{\nabla^{k-2}\BBB}_{L^p(M)}+C\sum_{\genfrac{}{}{0pt}{}
{s,r,t\in\N}{s+r+t=k-2}}\norma{\nabla^s\BBB\star\nabla^r\BBB\star\nabla^{t}\BBB}_{L^p(M)}\,.\label{gradkB}
\end{align}
Now, by Simons' identity~\eqref{codaz}, we have
$$
\nabla^{k-2}\Delta\BBB=\nabla^k\HHH+\nabla^{k-2}(\HHH\BBB^2)-\nabla^{k-2}(\vert\BBB\vert^2\BBB)\,,
$$
hence,
$$
\norma{\nabla^{k-2}\Delta\BBB}_{L^p(M)}\leqslant\norma{\nabla^k\HHH}_{L^p(M)}+C\sum_{\genfrac{}{}{0pt}{}
{s,r,t\in\N}{s+r+t=k-2}}\norma{\nabla^s\BBB\star\nabla^r\BBB\star\nabla^{t}\BBB}_{L^p(M)}\, .
$$
Using this estimate in inequality~\eqref{gradkB}, we conclude
\begin{align}
\norma{\nabla^k\BBB}_{L^p(M)}\leqslant &\,C\norma{\nabla^k\HHH}_{L^p(M)}+C\norma{\nabla^{k-2}\BBB}_{L^p(M)}+C\sum_{\genfrac{}{}{0pt}{}
{s,r,t\in\N}{s+r+t=k-2}}\norma{\nabla^s\BBB\star\nabla^r\BBB\star\nabla^{t}\BBB}_{L^p(M)}\,.
\end{align}
We now estimate any of the terms in the last sum as follows: we have
\begin{equation}\label{ppp1}
\norma{\nabla^s\BBB\star\nabla^r\BBB\star\nabla^{t}\BBB}_{L^p(M)}\leqslant\, C\norma{\nabla^s\BBB}_{L^{\alpha p}(M)}\norma{\nabla^r\BBB}_{L^{\beta p}(M)}\norma{\nabla^{t}\BBB}_{L^{\gamma p}(M)}\,,
\end{equation}
with 
$$
\alpha=\frac{k+1}{s+1}\,,\qquad\qquad
\beta=\frac{k+1}{r+1}\,,\qquad\qquad
\gamma=\frac{k+1}{t+1}\,,
$$
hence, $1/\alpha+1/\beta+1/\gamma=1$, as $s+r+t=k-2$. 
Moreover, using the interpolation estimates~\eqref{interpGN} (extended to tensors -- see Remark~\ref{tenrem}), there hold
\begin{align}
\norma{\nabla^s\BBB}_{L^{p\alpha}(M)} \leqslant &\,C \big(\norma{\nabla^k\BBB}_{L^p(M)}+\norma{\BBB}_{L^p(M)}\big)^{\overline{\theta}_\alpha}\norma{\BBB}_{L^n(M)}^{1-\overline{\theta}_\alpha}\\
\norma{\nabla^r\BBB}_{L^{p\beta}(M)} \leqslant &\,C \big(\norma{\nabla^k\BBB}_{L^p(M)}+\norma{\BBB}_{L^p(M)}\big)^{\overline{\theta}_\beta}\norma{\BBB}_{L^n(M)}^{1-\overline{\theta}_\beta}\\
\norma{\nabla^t\BBB}_{L^{p\gamma}(M)} \leqslant &\,C \big(\norma{\nabla^k\BBB}_{L^p(M)}+\norma{\BBB}_{L^p(M)}\big)^{\overline{\theta}_\gamma}\norma{\BBB}_{L^n(M)}^{1-\overline{\theta}_\gamma}
\end{align}
with $\overline{\theta}_\alpha=\frac{s+1}{k+1}$, $\overline{\theta}_\beta=\frac{r+1}{k+1}$ and $\overline{\theta}_\gamma=\frac{t+1}{k+1}$, determined by
\begin{align}
\frac{1}{p\alpha}=&\,\frac{s}{n} + \overline{\theta}_\alpha\Bigl(\frac{1}{p}-\frac{k}{n}\Bigr)
+\frac{1-\overline{\theta}_\alpha}{n}\\
\frac{1}{p\beta}=&\,\frac{r}{n} + \overline{\theta}_\beta\Bigl(\frac{1}{p}-\frac{k}{n}\Bigr)
+\frac{1-\overline{\theta}_\beta}{n}\\
\frac{1}{p\gamma}=&\,\frac{t}{n} + \overline{\theta}_\gamma\Bigl(\frac{1}{p}-\frac{k}{n}\Bigr)
+\frac{1-\overline{\theta}_\gamma}{n}.
\end{align}
Noticing that $\overline{\theta}_\alpha\in(s/k,1)$, $\overline{\theta}_\beta\in(r/k,1)$ and $\overline{\theta}_\gamma\in(t/k,1)$, if we choose $\theta_\alpha$, $\theta_\beta$ and $\theta_\gamma$ such that
$$
\frac{s}{k}<\theta_\alpha<\overline{\theta}_\alpha=\frac{s+1}{k+1},\quad\frac{r}{k}<\theta_\beta<\overline{\theta}_\beta=\frac{r+1}{k+1}\quad\text{and }\quad\frac{t}{k}<\theta_\gamma<\overline{\theta}_\gamma=\frac{t+1}{k+1}\,,
$$
respectively close to $\overline{\theta}_\alpha$, $\overline{\theta}_\beta$ and $\overline{\theta}_\gamma$, the uniquely determined values $q_\alpha$, $q_\beta$ and $q_\gamma$ satisfying
\begin{align}
\frac{1}{p\alpha}=&\,\frac{s}{n} + \theta_\alpha\Bigl(\frac{1}{p}-\frac{k}{n}\Bigr)
+\frac{1-\theta_\alpha}{q_\alpha}\\
\frac{1}{p\beta}=&\,\frac{r}{n} + \theta_\beta\Bigl(\frac{1}{p}-\frac{k}{n}\Bigr)
+\frac{1-\theta_\beta}{q_\beta}\\
\frac{1}{p\gamma}=&\,\frac{t}{n} + \theta_\gamma\Bigl(\frac{1}{p}-\frac{k}{n}\Bigr)
+\frac{1-\theta_\gamma}{q_\gamma}
\end{align}
must be close to $n$, thus properly choosing $\theta_\alpha$, $\theta_\beta$ and $\theta_\gamma$, as above, we have that $q_\alpha$, $q_\beta$ and $q_\gamma$ are smaller than $q>n$. Hence, by the interpolation estimates again, we have
\begin{align}
\norma{\nabla^s\BBB}_{L^{p\alpha}(M)} \leqslant &\,C \big(\norma{\nabla^k\BBB}_{L^p(M)}+\norma{\BBB}_{L^p(M)}\big)^{\theta_\alpha}\norma{\BBB}_{L^{q_\gamma}(M)}^{1-\theta_\alpha}\\
\norma{\nabla^r\BBB}_{L^{p\beta}(M)}\leqslant &\,C \big(\norma{\nabla^k\BBB}_{L^p(M)}+\norma{\BBB}_{L^p(M)}\big)^{\theta_\beta}\norma{\BBB}_{L^{q_\beta}(M)}^{1-\theta_\beta}\\
\norma{\nabla^t\BBB}_{L^{p\gamma}(M)} \leqslant &\,C \big(\norma{\nabla^k\BBB}_{L^p(M)}+\norma{\BBB}_{L^p(M)}\big)^{\theta_\gamma}\norma{\BBB}_{L^{q_\gamma}(M)}^{1-\theta_\gamma}\,.
\end{align}
Then, since $\norma{\BBB}_{L^{q_\alpha}(M)}$, $\norma{\BBB}_{L^{q_\beta}(M)}$ and $\norma{\BBB}_{L^{q_\gamma}(M)}$ are bounded by $C\norma{\BBB}_{L^q(M)}$, being the three exponents smaller that $q$ (the volumes are equibounded for all $M\in\mathfrak{C}^1_\delta(M_0)$), we get
\begin{align}
\norma{\nabla^s\BBB}_{L^{p\alpha}(M)}
\leqslant &\,C \big(\norma{\nabla^k\BBB}_{L^p(M)}+\norma{\BBB}_{L^p(M)}\big)^{\theta_\alpha}\norma{\BBB}_{L^{q}(M)}^{1-\theta_\alpha}\\
\norma{\nabla^r\BBB}_{L^{p\beta}(M)}
\leqslant &\,C \big(\norma{\nabla^k\BBB}_{L^p(M)}+\norma{\BBB}_{L^p(M)}\big)^{\theta_\beta}\norma{\BBB}_{L^{q}(M)}^{1-\theta_\beta}\\
\norma{\nabla^t\BBB}_{L^{p\gamma}(M)}
\leqslant &\,C \big(\norma{\nabla^k\BBB}_{L^p(M)}+\norma{\BBB}_{L^p(M)}\big)^{\theta_\gamma}\norma{\BBB}_{L^{q}(M)}^{1-\theta_\gamma}\,,
\end{align}
Letting
$$
\Theta=(\theta_\alpha+\theta_\beta+\theta_\gamma)<\frac{s+1}{k+1}+\frac{r+1}{k+1}+\frac{t+1}{k+1}=1\,,
$$
as $s+r+t=k-2$, putting these estimates in inequality~\eqref{ppp1} and recalling Theorem~\ref{GCZ0}, we conclude
\begin{align}
\norma{\nabla^s\BBB\star\nabla^r\BBB\star\nabla^{t}\BBB}_{L^p(M)}\!
\leqslant&\,C\big(\norma{\nabla^k\BBB}_{L^p(M)}\!+\norma{\BBB}_{L^p(M)}\big)^{\Theta}\,\norma{\BBB}_{L^q(M)}^{3-\Theta}\\
\leqslant&\,C\big(\norma{\nabla^k\BBB}_{L^p(M)}\!+\norma{\BBB}_{L^p(M)}\big)^{\Theta}\big(1+\norma{\HHH}_{L^q(M)}\big)^{3-\Theta}\\
\leqslant&\,C\big(\norma{\nabla^k\BBB}_{L^p(M)}\!+\norma{\BBB}_{L^p(M)}\big)^{\Theta}\,,\label{eqcar111}
\end{align}
with $C=C(M_0,p,\delta,\norma{\HHH}_{L^n(M)},\norma{\HHH}_{L^q(M)})=C(M_0,p,\delta,\norma{\HHH}_{L^q(M)})$, as $q>n$.\\
Hence, by means of Young inequality, as $\Theta<1$, we estimate
\begin{align}
\norma{\nabla^k\BBB}_{L^p(M)}\leqslant &\,C\norma{\nabla^k\HHH}_{L^p(M)}
\\&+C\norma{\nabla^{k-2}\BBB}_{L^p(M)}+
C\big(\norma{\nabla^k\BBB}_{L^p(M)}+\norma{\BBB}_{L^p(M)}\big)^\Theta\\
\leqslant &\,C\norma{\nabla^k\HHH}_{L^p(M)}
\\&+C\norma{\nabla^{k-2}\BBB}_{L^p(M)}+
C\varepsilon\norma{\nabla^k\BBB}_{L^p(M)}+C\norma{\BBB}_{L^p(M)}+C\,,
\end{align}
then choosing $\varepsilon>0$ such that $C\varepsilon<1/2$, after ``absorbing'' in the left hand side the term $C\varepsilon\norma{\nabla^k\BBB}_{L^p(M)}$ and estimating $\norma{\BBB}_{L^p(M)}$ with $C(1+\norma{\HHH}_{L^p(M)})$, we obtain
$$
\norma{\nabla^k\BBB}_{L^p(M)}\leqslant C\norma{\nabla^k\HHH}_{L^p(M)}+C\norma{\nabla^{k-2}\BBB}_{L^p(M)}+C\norma{\HHH}_{L^p(M)}+C\,.
$$
The term $\norma{\nabla^{k-2}\BBB}_{L^p(M)}$ can be treated analogously, by interpolation between $\norma{\nabla^k\BBB}_{L^p(M)}$ and $\norma{\BBB}_{L^p(M)}$ (it is actually easier to deal with it) and $\norma{\HHH}_{L^p(M)}\leqslant C(M_0,p,q,\delta)\norma{\HHH}_{L^q(M)}$, hence we finally have the desired estimate
$$
\norma{\nabla^k\BBB}_{L^p(M)}\leqslant C\norma{\nabla^k\HHH}_{L^p(M)}+C\,,
$$
with $C=C(M_0,p,q,\delta,\norma{\HHH}_{L^q(M)})$, for any $M\in \mathfrak{C}^1_\delta(M_0)$ with $\delta>0$ small enough.\\
If $p\in(1,+\infty)$, we argue as before, but using directly inequality~\eqref{c1}, which holds with a constant  $C=C(M_0,p,\delta,\Vert\BBB\Vert_{L^\infty(M)})$ and getting inequality~\eqref{eqcar111} with a constant $C=C(M_0,p,\delta,\Vert\BBB\Vert_{L^\infty(M)})$, by simply choosing a suitably large $q>n$ and estimating $\norma{\BBB}_{L^q(M)}$ with $C\norma{\BBB}_{L^\infty(M)}$. The rest of the proof goes in the same way, estimating all the terms $\norma{\BBB}_{L^q(M)}$ and $\norma{\HHH}_{L^q(M)}$ with $C\norma{\BBB}_{L^\infty(M)}$.
\end{proof}

\section{Other inequalities}

Let $M_0$ be a smooth and compact hypersurface embedded in $\R^{n+1}$, bounding a domain $E_0$ and $\eps>0$ the width of a tubular neighborhood $N_\eps$ of $M_0$. For any $\delta\in(0,\eps)$, we consider the family of domains 
\begin{equation}\label{psidescr3}
\mathcal{C}^1_\delta(E_0)=
\left\{E = \Psi(E_0)\,:\,
\begin{aligned}
&\, \text{ $\Psi:\overline{E}_0\to \overline E$ is a diffeomorphism with $\norma{\Psi-\mathrm{Id}}_{C^1(E_0)}<\delta$\ }\\
&\,\, \text{$\Psi(x)=x+\psi(x)\nu_0(x)$ for every $x\in M_0$ and $\norma{\psi}_{C^1(M_0)}<\delta$}
\end{aligned}\right\}
\end{equation}
where $\nu_0$ is the unit normal vector field pointing outward of $M_0$.\\
Then, the Jacobian of the map $\Psi :\overline{E}_0\to \overline E$ (and also the tangential one of its restriction to $M_0$) is bounded from above and from below by some constants which depend only on $\delta$ and the second fundamental form of $M_0$ (see Section~\ref{SobSec} for details).

It clearly follows that if $E\in \mathcal{C}^1_\delta(E_0)$, then $M=\partial E=\Psi(M_0)\in\mathfrak{C}^1_\delta(M_0)$. Moreover, if $M\in\mathfrak{C}^1_{\delta'}(M_0)$, then there exists a smooth function $\psi:M_0\to\R$ with $\norma{\psi}_{C^1(M_0)}<{\delta'}$, such that 
$M=\bigl\{x+\psi(x)\nu(x)\, : \, x\in M_0\bigr\}$, then we can construct a smooth diffeomorphism $\Psi:\overline{E}_0\to \overline E$ as follows ($E$ is the domain bounded by $M$):
$$
\Psi (x) = \begin{cases}
 x \, \, &\text{if $x \in E_0\setminus N_\eps$} \\
 x+ \zeta(d_0(x)/\varepsilon)\psi (\pi_0(x))\nabla^{\R^{n+1}}\!d_0(x) \,\,&\text{if $x \in \overline E_0\cap N_\eps$}
\end{cases} 
$$
where $d_0$ is the signed distance function from $M_0$ (which is negative in $E_0$) and $t\mapsto \zeta(t)$ is a smooth monotone nondecreasing function, defined on $\R$, such that it is equal to $1$ if $t\geqslant 0$ and to $0$ if $t\leqslant-1/2$, with $\vert\zeta'(t)\vert\leqslant 3$, for every $t\in\R$.
So, it follows
$$
\norma{\Psi-\mathrm{Id}}_{C^1(E_0)}=\norma{ \zeta(d_0(\cdot)/\varepsilon)\psi (\pi_0(\cdot))\nabla^{\R^{n+1}}\! d_0(\cdot)}_{C^1(\overline E_0\cap N_\eps)}\leqslant C(M_0,\varepsilon)\norma{\psi}_{C^1(M_0)}\,.
$$
Hence, fixed any $\delta\in(0,\varepsilon)$, depending the constant $C$ only on $M_0$ and $\varepsilon$, possibly choosing ${\delta'}$ small enough, the set $E$ belongs to $\mathcal{C}^1_\delta(E_0)$.

We now discuss some uniform inequalities involving also the domains bounded by the hypersurfaces.

\subsection{Trace inequalities}\ \\
Letting $E_0$, $M_0$, $\varepsilon>0$ and $\delta>0$ as above and any $E\in \mathcal{C}^1_\delta(E_0)$ (with associated smooth diffeomorphism $\Psi:\overline{E}_0\to \overline E$), it is well known that the {\em trace} of any function $u \in H^1(E)$ (a real function on $M=\partial E$, which we still simply denote by $u$, that coincides with the restriction of $u$ to $M$, if $u\in C^0(\overline E)$) is well defined and that the following {\em trace inequality} holds (see~\cite[Chapter~4, Proposition~4.5]{Taylor1}),
\begin{equation}\label{1E_0}
\|u\|^2_{H^{1/2}(M)}\leqslant C_E\int_{E}u^2+|\nabla u|^2\, dx\,,
\end{equation}
which implies
\begin{equation*}
\|u-\widetilde u\|^2_{H^{1/2}(M)}\leqslant C_E\int_{E}|\nabla u|^2\, dx\,,
\end{equation*}
where $\widetilde{u}=\fint_Eu\,dx$ (see also~\cite{Ev,Leoni}).\\
We want to show that these inequalities hold with uniform constants $C(M_0,\delta)$, for every $E\in \mathcal{C}^1_\delta(E_0)$.
Expressing $\|u\|^2_{H^{1/2}(M)}$ by means of the Gagliardo $W^{1/2,2}$--seminorm of a function $u\in L^2(M)$ and setting $\Phi=\Psi|_{M_0}:M_0\to M$, we have 
\begin{align}
\|u\|^2_{H^{1/2}(M)}=&\,\|u\|_{L^2(M)}+[u]_{W^{1/2,2}(M)}^2\\
=&\,\|u\|_{L^2(M)}+\int_{M}\int_{M}\frac{|u(y)-u(y^*)|^2}{|y-y^*|^{n+1}}\,d\mu(y)\,d\mu(y^*)\\
\leqslant&\,C\|u\comp\Phi\|_{L^2(M_0)}+\int_{M_0}\int_{M_0}\frac{|u(\Phi(x))-u(\Phi(x^*))|^2}{|\Phi(x)-\Phi(x^*)|^{n+1}}\,J\Phi(x)J\Phi(x^*)\,d\mu_0(x)d\mu_0(x^*)\\
\leqslant&\,C\|u\comp\Psi\|_{L^2(M_0)}+C\int_{M_0}\int_{M_0}\frac{|u(\Psi(x))-u(\Psi(x^*))|^2}{|x-x^*|^{n+1}}\,d\mu_0(x)d\mu_0(x^*)\\
\leqslant&\,C_{E_0}\int_{E_0}|u(\Psi(x))|^2+|\nabla^0(u\comp\Psi(x))|^2\, dx\\
\leqslant &\,C\int_{E}u^2+|\nabla u|^2\, dx=C\|u\|^2_{H^1(E)}\,,\label{eqcar444}
\end{align}
where the constant $C$ depends only on $E_0$ (we applied inequality~\eqref{1E_0} for $E_0$ in passing from the fourth to the fifth line) and $\delta$ (in bounding $|d\Psi|$, $|d\Phi|$, $J\Psi$ and $J\Phi$ above and below away from zero).

\begin{remark}
With a similar argument, we can show the following generalization of this inequality, with a uniform constant
\begin{equation}\label{traceineq}
\Vert u \Vert _{H^{s-1/2}(M)} \leqslant C(E_0,s,\delta)\Vert u\Vert _{H^s(E)}
\end{equation}
(see again~\cite[Chapter~4, Proposition~4.5]{Taylor1}), for $s\in(1/2,3/2)$.
\end{remark}

\subsection{Inequalities for harmonic extensions}\ \\
We let $E_0$, $M_0$, $\varepsilon>0$ and $\delta>0$ as above and $E\in \mathcal{C}^1_\delta(E_0)$ (with associated smooth diffeomorphism $\Psi:\overline{E}_0\to \overline E$), with $M= \pa E\in\mathfrak{C}^1_\delta(M_0)$.\\
We denote by $u:E\to\R$ the harmonic extension of a function $f:M\to\R$ in $H^{1/2}(M)$ to $E$. We aim to show that the following inequality (see~\cite[Chapter~5, Proposition~1.7]{Taylor1})
\begin{equation}\label{extineq}
\Vert u \Vert_{H^{1}(E)} \leqslant C_{E}\Vert f \Vert_{H^{1/2} (M)}\,,
\end{equation}
which implies
\begin{equation}
\int_{E}|\nabla u|^2\, dx \leqslant C_{E}\Vert f\Vert^2_{H^{1/2}(M)}\,,
\end{equation}
for every $E\in \mathcal{C}^1_\delta(E_0)$, with uniform constants $C=(E_0,\delta)$.\\
Arguing as above, in formula~\eqref{eqcar444}, we end up with the following inequalities:
\begin{align}
\Vert u \Vert_{H^{1}(E)} \leqslant&\, C(E_0,\delta) \Vert u\comp\Psi \Vert_{H^{1}(E_0)}\\
\Vert u\comp\Psi\Vert_{H^{1}(E_0)} \leqslant&\, C_{E_0}\Vert f\comp\Psi \Vert_{H^{1/2}(M_0)}=C_{E_0}\Vert f\comp\Phi \Vert_{H^{1/2}(M_0)}\\
\Vert f\comp\Phi \Vert_{H^{1/2}(M_0)} \leqslant&\, C(M_0,\delta) \Vert f \Vert _{H^{1/2}(M)}
\end{align}
where the second estimate is given by inequality~\eqref{extineq} for $E_0$. Putting them together, we have the conclusion.

\begin{remark}
As above, we have the following generalization, for $s\in[1/2,3/2)$,
\begin{equation}\label{extineq2}
\Vert u \Vert_{H^{s+1/2}(E)} \leqslant C(E_0,s,\delta)\Vert f \Vert_{H^{s} (M)}
\end{equation}
(see again~\cite[Chapter~5, Proposition~1.7]{Taylor1}).
\end{remark}

\section{Some remarks}

We collect here some remarks about the conclusions of the previous sections.

\begin{itemize}

\item All the constants depend on the geometric properties of $M_0$, in particular on the maximal width of a tubular neighborhood, its volume and its second fundamental form. Hence, uniformly controlling such quantities gives uniform estimates for larger families of hypersurfaces, 
see~\cite{breun1,breun2,breun3,della1,langer2} for a deeper and detailed discussion).

\item Notice that for Sobolev, Poincar\'e, interpolation, trace and ``harmonic extension'' inequalities, we do not ask $\delta>0$ to be small, but just $\delta<\eps$, while for the Calder\'on--Zygmund--type inequalities, that we worked out in Section~\ref{CZsec}, a smallness condition on $\delta$ is necessary for the conclusions.

\item All the inequalities holds uniformly also for families of immersed--only hypersurfaces (non necessarily embedded), if they can be expressed as graphs on a fixed compact, smooth hypersurface, possibly immersed--only too.

\item It is easy to see that everything we did still works also if the ambient is a {\em flat}, complete Riemannian manifold, in particular in any flat torus $\T^n$. With some effort, the results can be generalized to graph hypersurfaces in any complete Riemannian manifold, then the constants also depends on the geometry (in particular, on the curvature) of such an ambient space. 

\end{itemize}

\bibliographystyle{amsplain}
\bibliography{UniformIneqGraph}

\end{document}